\documentclass[11pt]{amsart}

\usepackage{amscd,amsthm,amsfonts,amssymb,amsmath,esint}
\usepackage{mathrsfs}
\usepackage[colorlinks=true]{hyperref}
\usepackage[all]{xy}
\usepackage{CJK}
\usepackage{comment}

    \newcommand{\BA}{{\mathbb {A}}} 
    \newcommand{\BC}{{\mathbb {C}}} 
     \newcommand{\BF}{{\mathbb {F}}}
     \newcommand{\BH}{{\mathbb {H}}}

     \newcommand{\BP}{{\mathbb {P}}}
    \newcommand{\BQ}{{\mathbb {Q}}}

     \newcommand{\BZ}{{\mathbb {Z}}}

    \newcommand{\CE}{{\mathcal {E}}} 
     \newcommand{\CH}{{\mathcal {H}}}
     \newcommand{\CJ}{{\mathcal {J}}}

    \newcommand{\CO}{{\mathcal {O}}}

     \newcommand{\CX}{{\mathcal {X}}}
    \newcommand{\CY}{{\mathcal {Y}}}

    \newcommand{\fa}{{\mathfrak{a}}} 
     
     \newcommand{\ff}{{\mathfrak{f}}}
    \newcommand{\fg}{{\mathfrak{g}}}

    \newcommand{\fm}{{\mathfrak{m}}} 
     \newcommand{\fp}{{\mathfrak{p}}}
    \newcommand{\fq}{{\mathfrak{q}}}

    \newcommand{\fU}{{\mathfrak{U}}}

\newcommand{\ab}{{\mathrm{ab}}}                     

\newcommand{\Gal}{{\mathrm{Gal}}}                   
\newcommand{\GL}{{\mathrm{GL}}}                     
\newcommand{\Norm}{{\mathrm{Norm}}}                 
\newcommand{\Spec}{{\mathrm{Spec}\,}}               

\newcommand{\Spf}{{\mathrm{Spf}}}                   
\newcommand{\tor}{\mathrm{tor}}                     
\newcommand{\tr}{{\mathrm{tr}}}                     

\renewcommand{\mod}{\, \mathrm{mod}\, }

\font\cyr=wncyr10 \newcommand{\Sha}{\hbox{\cyr X}}

\newcommand{\sk}{\medskip}
\newcommand{\s}{\sk\noindent}

\newcommand{\ov}{\overline}

\newcommand{\wh}{\widehat}

\newcommand{\ra}{\rightarrow}                           

\newcommand{\pair}[1]{\left\langle {#1}\right \rangle}             
\newcommand{\lrb}[1]{\left(#1\right)}                   
\newcommand{\set}[1]{\left\{#1\right\}}                 

\newtheorem{thm}{Theorem}[section]

\newtheorem{theorem}[thm]{Theorem}

\newtheorem{corollary}[thm]{Corollary}

\newtheorem{lemma}[thm]{Lemma}

\newtheorem{prop}[thm]{Proposition}
\newtheorem{proposition}[thm]{Proposition}

\theoremstyle{definition}

\theoremstyle{remark}

\newtheorem{remark}[thm]{Remark}

\numberwithin{equation}{subsection}

\def\mat(#1,#2,#3,#4){\begin{pmatrix}#1 & #2 \\ #3 & #4\end{pmatrix}}
\newcommand{\matrixx}[4]{\begin{pmatrix}#1 & #2 \\ #3 & #4\end{pmatrix} }

\newcommand{\hilbert}[4]{\left(\frac{#1,#2}{#3}\right)_{#4}}

\begin{document}
\title{A proof of the $4,7$ cases of Sylvester's conjecture on cube sums}
\author{Hongbo Yin}

\begin{abstract}
In this paper, we prove that every prime $p$ which is congruent to $4,7$ modulo $9$ is the sum of two rational cubes. This is $2/3$ of Sylvester's conjecture which has a history of nearly 150 years since 1879. In the proof, we use recent progress on Full BSD conjecture of rank $0$ elliptic curves in \cite{BF} to deduce that the Manin-Stevens constants of some families of elliptic curves are units. We also use recent solutions of Unbounded Denominators Conjecture in \cite{CDT} to prove that some cubic roots of modular functions are invariant under some congruence subgroups. Instead of using the Unbounded Denominators Conjecuture, we also give another conditional proof assuming the GRH for number fields or Artin's primitive root conjecture for arithmetic progressions.
\end{abstract}

\address{School of Mathematics, Shandong University, Jinan, Shandong,  250100,
China}
\email{yhb@sdu.edu.cn}

\maketitle

\section{Introduction}

The Sylvester problem, i.e. which integer can be written as the sum of two rational cubes (we will call it cube sum for short) has been investigated intensively recently, see \cite{CST17}\cite{DV17}\cite{HSY19}\cite{SSY}\cite{Yin1}\cite{Yin}\cite{ABS}\cite{MS1}\\ \cite{MS} and many other papers on the arXiv for example. This problem has a long history from Euclid, through Diophantine, Fermat, Euler, Dirichlet, and continued to Sylvester (1879) and Selmer (1951) who studied it dramatically. For this interesting history, please see Reyna's wonderful blog \cite{Reyna}. From the modern point of view, this problem is closely related to the elliptic curves $y^2=x^3-432n^2$ \cite{DV1}. This is very similar to the famous congruent number problem which is closely related to the elliptic curves $y^2=x^3-n^2x$ \cite{tian}. These two families of elliptic curves are also the first examples studied by Birch and Swinnerton-Dyer for their famous conjecture. The classical and typical result in the cube sum problem is Sylvester's theorem that if $p\equiv 2,5\mod 9$, then $p$ is not a cube sum. Contrary to this, Birch and Stephan \cite{BS} proposed the conjecture that if $p\equiv 4,7,8\mod 9$, then $p$ should be a cube sum based on the computation of root numbers of the corresponding elliptic curves and the same story also holds for $p^2$. This conjecture is also indicated (but not written down explicitly) by Selmer in \cite{Selmer51}, where he proved the upper bound of the algebraic rank of the elliptic curves is one. However, this conjecture is usually called Sylvester's conjecture, maybe because people believe that Sylvester should have thought this conjecture by his result although he did not write it down. 

In 1993, Elkies claimed a proof of this conjecture for primes $p\equiv 4,7\mod 9$. But he had never published any details about his proof. He only very briefly sketched his principle of constructing the (mock) Heegner points in \cite{Elkies}. Nobody else knows any further information about his construction. So the only known result on this conjecture is some partial result in all cases, see \cite{DV17}\cite{Yin}. In this paper, inspired by Elkies's sketch, we give a detailed construction (but not the same as Elkies's construction) of CM points 
which give non-torsion rational points on the elliptic curves for $p\equiv 4,7\mod 9$ and prove the following theorem:
\begin{theorem}\label{main}
Let $p\equiv 4,7\mod 9$ be a prime, both $p$ and $p^2$ are sums of two rational cubes. 
\end{theorem}
Let us describe our CM point briefly. Our construction gives an explicit algorithm for computing the rational solutions of the theorem. In this paper, we assume $i=1,2$\footnote{I am sorry for the abuse of $i$ but I feel $i$ is the simplest candidate for the superscript. Sometimes, $i$ will also be used as the imaginary number for example in the expressions $2\pi i$ or $i\infty$. But it will be easy to distinguish from the context.}.
For convenience, we will use the elliptic curve $E_{p^i}:y^2=x^3+p^{2i}/4$ which is isogenous to $y^2=x^3-432p^{2i}$ over $\BQ$. Let $K=\BQ(\sqrt{-3})$ be the CM field of $E_{p^i}$ and fix a splitting of $p=\varpi\bar\varpi$ with $\varpi\equiv 1\mod 3$ in $K$. By a theorem of Shimura, we have modular parametrizations 
$$\varphi: X_\Gamma\longrightarrow E_{\bar{\varpi}^i}:y^2=x^3+\bar{\varpi}^{2i}/4,$$
$$\varphi^c: X_\Gamma\longrightarrow E_{{\varpi}^i}:y^2=x^3+{\varpi}^{2i}/4$$
where
$$\Gamma=\set{\matrixx{a}{b}{c}{d}\in \Gamma_0(N): d \mod p\ \text{is a cube}}$$
and $N=9p$ or $27p$ depending on $p^i\equiv 4$ or $7$ modulo $9$. 
We also consider the maps
\[\phi: E_{\bar\varpi^i}\longrightarrow E_{p^i},\]
\[\phi^c: E_{\varpi^i}\longrightarrow E_{p^i}\]
given by
\[(x,y)\mapsto (\sqrt[3]{\varpi^{2i}}x, \varpi^i y),\]
\[(x,y)\mapsto (\sqrt[3]{\bar\varpi^{2i}}x, \bar\varpi^i y).\]
Let $r\in\BZ$ be a solution of $r^2-r+1\equiv 0\mod 3p$ such that $-r\equiv \omega^2\mod\varpi$ where $\omega=-\frac{1}{2}+\frac{\sqrt{-3}}{2}$ is a cubic root of unity. Let $\tau_r=\frac{-1}{3(\omega+r)}$ be the CM point on the upper half plane, then using Shimura's reciprocity law we can prove $\phi\circ\varphi(\tau_r), \phi^c\circ\varphi^c(\tau_r)\in E_{p^i}(K)$.

The work of this paper is very different from previous results in the construction of CM points, e.g.,\cite{DV17}\cite{HSY19}\cite{Yin}.  In traditional method, one fixes a modular parametrization of a chosen elliptic curve over $\BQ$. The points constructed are defined over the ring class field of some imaginary quadratic field and one need to take a trace. Instead, we use the modular parametrization of varied elliptic curves over imaginary quadratic fields and our points are defined over the base field directly, so we do not take any trace. To prove the point constructed above is non-torsion, we write $\varphi(\tau)=(x(\tau),y(\tau))$ and $\varphi^c(\tau)=(x^c(\tau),y^c(\tau))$. We first prove the functions 
$$F_-(z)=\lrb{y(z)-\frac{\bar\varpi^i}{2}}^{\frac{1}{3}}\lrb{y^c(z)-\frac{\varpi^i}{2}}^{-\frac{1}{3}},$$
$$F_+(z)=\lrb{y(z)+\frac{\bar\varpi^i}{2}}^{\frac{1}{3}}\lrb{y^c(z)+\frac{\varpi^i}{2}}^{-\frac{1}{3}},$$ 
$$H(z)=\lrb{y(z)-\frac{\bar\varpi^i}{2}}^{\frac{1}{3}}\lrb{y^c(z)+\frac{\varpi^i}{2}}^{\frac{1}{3}},$$ 
$$G(z)=\lrb{y(z)+\frac{\bar\varpi^i}{2}}^{\frac{1}{3}}\lrb{y^c(z)-\frac{\varpi^i}{2}}^{\frac{1}{3}}$$ are $\Gamma$ invariant. This implies $F_-(\tau_r)$, $F_+(\tau_r)$, $H(\tau_r)$, $G(\tau_r)\in K(\sqrt[3]{\varpi})$ again by Shimura reciprocity law. Then we prove that without loss of generality we can assume both $\varphi(\tau_r)$ and $\varphi^c(\tau_r)$ are not zero points. As a result, if both $\varphi(\tau_r)$ and $\varphi^c(\tau_r)$ are torsion points, then at least one of $F_-(\tau_r)$, $F_+(\tau_r)$, $H(\tau_r)$, $G(\tau_r)$ has a factor $\sqrt[3]{\bar\varpi}$, which is a contradiction. 

In the first step, in order to use the Unbounded Denominators Conjecture which is recently proved by Calegari-Dimitrov-Tang \cite{CDT} to prove that the functions above are modular on some congruence subgroups, we need to show the modular parametrizations have integral coefficients. For this we need first to know that the Manin-Stevens constants of $E_{\bar\varpi^i}$ and $E_{\varpi^i}$ are units. At  unramified places we can use the results or methods of Cesnavicius \cite{Ce,Ce1}. At the semistable places above $p$, Cesnavicius's results do not apply since our elliptic curve is not defined over $\BQ$. Instead, we approach this by comparing the periods of twists of modular forms. At the most complicated place $3$, this is achieved using  the Full BSD conjecture of $E_{\bar\varpi^i}$ and $E_{\varpi^i}$ due to Burungale-Flach \cite{BF}. This is a totally new approach to the Manin-Stevens constant. The utility of the most recent progress in number theory illustrates that our proof of the main result is totally original and new. 

We will also give another short and direct proof of the first step assuming the GRH or the generalized Artin's primitive root conjecture. But we still need to know some Manin-Stevens constants are units first. It is a little surprising that Sylvester's conjecture can be deduced from these two at first sight irrelevant conjectures.

This paper is organized as follows. In section 2, we prepare some results we will need. In section 3, we investigate the modular curve and Hecke action. In section 4, we construct the CM point. In section 5, we study the cubic root of modular functions. In section 6, we study the Manin-Stevens constant. In section 7,  we prove the integrality of the modular parametrization. In section 8, we prove the main result. In section 9, we give another proof assuming the generalized Artin primitive root conjecture or GRH. In section 10, we give some examples.

\subsection*{Acknowledgments} The author would like to thank Professor Ye Tian for introducing this beautiful conjecture to him in 2016 when the author was a postdoctoral fellow at Morningside Center of Mathematics. He also wants to thank John Voight, Hang Liu, Bin Guan for useful discussions. The author is grateful to Brian Conrad, Frank Calegari and Kestutis Cesnavicius for the discussion and help on the integrality of modular parametrization. The author also thanks Frank Calegari, Vesselin Dimitrov, Yunqing Tang for the explanation of their result in the meromorphic modular functions case. Many of the computations in the research are done on the Sagemath system \cite{Sage}, and thanks are also given to the Sagemath team.

\section{Preliminary}
First of all, let us prove a result concerning the ray class field extensions of the imaginary quadratic field $K=\BQ(\sqrt{-3})$. For an ideal $\fm$ of $\CO_K$, let $H_{\fm}$ be the ray class field of $K$ with modulus $\fm$. Let $\sigma: \wh{K}^\times\ra \Gal(K^\ab/K)$ be the Artin reciprocity law and we denote by $\sigma_t$ the image of $t\in\wh{K}^\times$. Write $\fm=\prod_v\fp_v ^{n_v}$ and set
\[U_\fm=\prod _{v|\fm}\lrb{1+\fp_v ^{n_v}}\times\prod _{v\nmid\fm}\CO_{K,v} ^{\times},\]
Then $\Gal(H_\fm/K)\cong \widehat{K}^\times/KU_\fm$ by class field theory.

In the rest of the paper, $\omega=(-1+\sqrt{-3})/2$ is a cubic root of unity. The letter $p$ will always mean a prime integer congruent to $4,7\mod 9$ and we will use $q$ to denote a general prime. Fix a decomposition $p=\varpi\bar{\varpi}$ in $K$ such that $\varpi\equiv 1\mod 3$ and assume $\fp=(\varpi)$ and $\bar{\fp}=(\bar\varpi)$ are the prime ideals above $p$.
For an element $a$ of $K$, we will use $a_v$ to denote the embedding of $a$ into $\wh{K}^\times$ with the $v$-place $a$ and all other places $1$ (but we will use $a_3$ instead of $a_{\sqrt{-3}}$ below for simplicity, e.g., we write $\omega_3$ for $\omega_{\sqrt{-3}}$). Since $K(\sqrt[3]{\varpi})$ is not Galois over $\BQ$, the field $K(\sqrt[3]{\varpi})$ is not contained in any ring class field by the result in \cite{Cox89}.
\begin{prop}\label{LCF}
Let the notation be as above, we have
\begin{itemize}
\item[1.] The field $H_{9\varpi}=H_{3\varpi}(\sqrt[3]{3})$ with Galois group 
$$\Gal(H_{9\varpi}/H_{3\varpi})\simeq \langle1+3\omega_3\rangle^{\BZ/3\BZ},$$ 
and
\[\left(\sqrt[3]{3}\right)^{\sigma_{1+3\omega_3}-1}=\omega^2,\ \ \ \left(\sqrt[3]{\varpi}\right)^{\sigma_{1+3\omega_3}-1}=1.\]
\item[2.] The field $K(\sqrt[3]{\varpi})$ is contained in $H_{3\varpi}$.
\item[3.] We have 
\[\left(\sqrt[3]{\varpi}\right)^{\sigma_{\omega_\fp}-1}=\left\{\begin{aligned} {\omega},&\quad p\equiv 4\mod 9, \\{\omega^2},&\quad p\equiv 7\mod 9.  \end{aligned}\right.\]
So, $\Gal(K(\sqrt[3]{\varpi})/K)=\pair{\sigma_{\omega_\fp}}$.
\item[4.] All the above also holds if we change $\varpi$ to $\bar\varpi$.
\end{itemize}
\end{prop}
\begin{proof}
For the first assertion, the Galois group
\[\Gal(H_{9\varpi}/H_{3\varpi})\simeq K^\times U_{3\varpi}/K^\times U_{9\varpi}\]
is cyclic of order $3$ and generated by $1+3\omega_3$. Let $v$ be the place corresponding to the prime ideal $(1+3\omega)$. Then by the local-global principle, we have
\[\left(\sqrt[3]{3}\right)^{\sigma_{1+3\omega_3}-1}=\hilbert{1+3\omega_3}{ 3}{K_3}{3}=\hilbert{1+3\omega_v}{ 3}{K_v}{3}^{-1}=3^{-2}\mod (1+3\omega)=\omega^2,\]
here $\hilbert{\cdot}{\cdot}{K_w}{3}$ denotes the $3$rd Hilbert symbol over the local field $K_w$.
If $p\equiv 4\mod 9$, then $\varpi\equiv 3\omega+4,-3\omega+1,-2\mod 9\CO_K$; if $p\equiv 7\mod 9$, then $\varpi\equiv -3\omega-2,3\omega+1,4\mod 9\CO_K$. So

\[\left(\sqrt[3]{\varpi}\right)^{\sigma_{1+3\omega_3}-1}=\hilbert{1+3\omega_3}{ \varpi}{K_3}{3}=\hilbert{1+3\omega_v}{\varpi\mod 9}{K_v}{3}^{-1}=1.\]

For the second assertion, by the class field theory, it suffices to prove that, under the Artin reciprocity map, 
$$U_{3\varpi}=(1+3\CO_{K,3})(1+\varpi\CO_{K,\varpi})\prod_{v\nmid 3\varpi}\CO_{K,v}^\times$$ 
fixes $\sqrt[3]{\varpi}$.  Since $K(\sqrt[3]{\varpi})/K$ is unramified outside $3\varpi$, $\prod_{v\nmid 3\varpi}\CO_{K,v}^\times$ fixes $\sqrt[3]{\varpi}$. Using the Hilbert symbol, it is clear that $(1+\varpi\CO_{K,\varpi})$ fixes $\sqrt[3]{\varpi}$. Since $1+9\CO_{K,3}\subset (K^\times_3)^3$, Hilbert symbol shows $1+9\CO_{K,3}$ fixex $\sqrt[3]{\varpi}$. We have
\[(1+3\CO_{K,3})/(1+9\CO_{K,3})\cong \pair{1+3}^{\BZ/3\BZ}\times \pair{1+3\sqrt{-3}}^{\BZ/3\BZ}\]
But the first assertion shows $1+3\omega_3$ also fixes $\sqrt[3]{\varpi}$. By the reciprocity law for the $n$-th power residues \cite[Theorem 8.14]{Cox89} and the reciprocity law for the cubic residue symbol \cite[Theorem 4.12]{Cox89}, we have
$$\hilbert{2}{\varpi}{K_3}{3}=\hilbert{-2}{\varpi}{K_3}{3}=1.$$ So $1+3\CO_{K,3}$ fixes $\sqrt[3]{\varpi}$.


For the last assertion, 
\[\left(\sqrt[3]{\varpi}\right)^{\sigma_{\omega_\fp}-1}=\hilbert{\omega_\fp}{\varpi}{K_\fp}{3}=\omega^{\frac{p-1}{3}}=\left\{\begin{aligned} {\omega},&\quad p\equiv 4\mod 9, \\{\omega^2},&\quad p\equiv 7\mod 9.  \end{aligned}\right.\]
\end{proof}

Next, we will also discuss the Hecke characters of the elliptic curves with $j$-invariant $0$ over $K$. More explicitly, we will prove the following result:
\begin{theorem}\label{cmcharacter}
Let $D\in\CO_K$. Consider the elliptic curve $E: y^2=x^3+\frac{D}{4}$. For a prime ideal $\fq$ of $\CO_K$, we can write $\fq=(\varpi_\fq)$ with $\varpi_\fq\equiv 2\mod 3$. Assume $\fq\nmid 6 D$, then
\[\sharp \tilde{E}(\BF_\fq)=\sharp  \BF_\fq +1+\ov{\lrb{\frac{D}{\varpi_\fq}}_6}\varpi_\fq+\lrb{\frac{D}{\varpi_\fq}}_6\bar\varpi_\fq.\]
Where $\BF_\fq$ is the residue field of $\fq$ and $\lrb{\frac{\cdot}{\cdot}}_6$ is the $6$-th power residue symbol. 
\end{theorem}

The proof is almost the same as \cite[Theorem 4, Section 18.3]{IRbook} which deals with the case $D\in\BZ$.
For the convenience of the reader, we reproduce it here. We need some preparations.

Let $q$ be a prime. For $a\in\BF_{q^m}$, let $N_{q^m}(x^n=a)$ denote the number of solutions of the equation $x^n=a$ in $\BF_{q^m}$. If $d=gcd(n, q^m-1)$, we have 
\begin{proposition}\label{sumcharacter}
$N_{q^m}(x^n=a)=\sum_{\chi^d=1}\chi(a)$ where the sum is over all characters of $\BF_{p^m}$ of order dividing $d$.
\end{proposition}
\begin{proof}
For the formula, see \cite[Page 2]{Weil49}. For a detailed proof of the case $m=1$, see \cite[Proposition 8.1.5]{IRbook}.
\end{proof}

\begin{lemma}\label{Jacobi}
Assume $\rho$ is a character of $\BF_{q^m}^*$ of order 2 and $\xi$ any nontrivial character of $\mathbb{F}_{q^m}^{*}$. We also set $\rho(0)=\xi(0)=0$. Then the Jacobi sum $J(\rho,\xi)=\sum_{u+v=1}\rho(u)\xi(v)$ where $u,v$ runs over $\BF_{q^m}$ satisfies
$$J(\rho, \xi)=\xi(4) J(\xi, \xi).$$
\end{lemma}
\begin{proof}We have
$$
\begin{aligned}
J(\rho, \xi) &=\sum_{u+v=1} \rho(u) \xi(v) \\
&=\sum_{u+v=1}(1+\rho(u)) \xi(v)=\sum_{u+v=1} N_{q^m}\left(t^{2}=u\right) \xi(v) \\
&=\sum_{t} \xi\left(1-t^{2}\right)=\xi(4) \sum_{t} \xi\left(\frac{1-t}{2}\right) \xi\left(\frac{1+t}{2}\right)=\xi(4) J(\xi, \xi).
\end{aligned}
$$
\end{proof}

\begin{lemma}\label{Jacobi1}
Let $q\equiv 2\mod 3$ be a prime, then $J(\chi_q,\chi_q)=q$. Here $\chi_q(a)=(a / q)_{3}$ is the cubic residue symbol of $\CO_K$ viewed as a character of $\BF_{q^2}$.
\end{lemma}
\begin{proof}
It is well-known that $J(\chi_q,\chi_q)\ov{J(\chi_q,\chi_q)}=q^2$. We will prove that $J(\chi_q,\chi_q)\equiv 2\mod 3$, this will force $J(\chi_q,\chi_q)=q$ since there is only one element in 
$$\set{(-1)^i\omega^jq:i=0,1; j=0,1,2}$$ 
which is congruent $2$ modulo $3$. 

Let $g(\chi_q)=\sum_{t\in\BF_{q^2}}\chi_q(t)\zeta_q^{\tr(t)}$ be the Gauss sum of $\chi_q$. Then $g(\chi_q)g(\ov{\chi_q})=\chi_q(-1)q^2=q^2$.
Since both $\chi_q$ and $\chi_q^2$ are not the trivial character, by the basic property of Jacobi sums \cite[Theorem 1 and its Corollary, Page 93]{IRbook}, we have
$$J(\chi_q,\chi_q)=\frac{g(\chi_q)^2}{g(\chi_q^2)}=\frac{g(\chi_q)^3}{q^2}.$$ 
But 
\[g(\chi_q)^{3}=\left(\sum_{t\in\BF_{q^2}} \chi_q(t) \zeta_q^{\tr(t)}\right)^{3} \equiv \sum_{t\in\BF_{q^2}} \chi_q(t)^{3} \zeta_q^{3\tr(t)}=-1\mod 3,\]
so
$$
q^2 J(\chi_q, \chi_q)= g(\chi_q)^{3}\equiv-1\mod 3
$$
which implies $J(\chi_q, \chi_q)\equiv -1\mod 3$ since $q\equiv 2\mod 3$.
\end{proof}

\begin{proof}[Proof of Theorem \ref{cmcharacter}]
If $\varpi_\fq\bar\varpi_\fq=q$ with $q\equiv 1\mod 3$, then the proof in \cite[Page 305]{IRbook} also works here, since $\BF_{\fq}=\BF_q$ in this case. 

If $\varpi_\fq=q$ with $q\equiv 2\mod 3$, then $\BF_\fq=\BF_{q^2}$. By Proposition \ref{sumcharacter},

\[\begin{aligned}
N_{q^2}\left(y^{2}=x^{3}+\frac{D}{4}\right) &=\sum_{u+v=\frac{D}{4}} N_{q^2}\left(y^{2}=u\right) N_{q^2}\left(x^{3}=-v\right) \\
&=\sum_{u+v=\frac{D}{4}}(1+\rho(u))\left(1+\chi(-v)+\chi^{2}(-v)\right) \\
&=q^2+\sum_{u+v=\frac{D}{4}} \rho(u) \chi(v)+\sum_{u+v=\frac{D}{4}} \rho(u) \chi^{2}(v)
\end{aligned}\]
where $\rho$ is a character of order $2$ and $\chi$ a character of order $3$ of $\BF_{q^2}^*$ (we also set $\rho(0)=\chi(0)=0$). Making the substitutions $u=Du'/4$ and $v=Dv'/4$ we find 
\begin{equation}\label{ecff}
\sharp \tilde{E}(\BF_\fq)=q^2+1+\rho \chi(D/4) J(\rho, \chi)+\ov{\rho \chi(D/4)}\ov{J(\rho, \chi)},
\end{equation}
here $1$ is due to the point at $\infty$ and $J(\rho,\chi)=\sum_{a+b=1}\rho(a)\chi(b)$ is the Jacobi sum for $\BF_{q^2}$.
Using Lemma \ref{Jacobi}, Equation (\ref{ecff}) can be transformed into
\begin{equation}\label{eq1}
\sharp \tilde{E}(\BF_\fq)=q^2+1+\rho(D/4) \chi(D) J(\chi, \chi) + \overline{\rho(D/4)\chi(D) J(\chi, \chi)}.
\end{equation}
We want to specify $\rho$ and $\chi$. Let $(a / q)_{6}$ be the sixth power residue symbol and take $\rho(a)=(a / q)_{6}^{3}$ and $\chi(a)=(a / q)_{6}^{2}=(a / q)_{3} .$ Then $\rho \chi(a)=\rho(a) \chi(a)=(a / q)_{6}^{5}=(\overline{a / q})_{6} .$ Finally, Lemma \ref{Jacobi1} shows that $J\left(\chi_{q}, \chi_{q}\right)=q$. Substituting this information into Equation (\ref{eq1}), we finish the proof of the theorem.
\end{proof}

\newcommand{\varpii}{{\varpi^i}}
\newcommand{\bvarpi}{{\bar{\varpi}^{i}}}

By Theorem \ref{cmcharacter} as in \cite[Example 10.6]{Silvermanbook2}, the Hecke character $\psi_D$ associated to the elliptic curve $E$ in Theorem \ref{cmcharacter} satisfies $$\psi_D(\fq)=\ov{\lrb{\frac{D}{\varpi_\fq}}_6}\varpi_\fq$$ for any prime ideal $\fq\nmid 6D$ if we choose $\fq=(\varpi_\fq)$ with $\varpi_\fq\equiv 1\mod 3$. 
In this paper, we will consider the cubic twist families of elliptic curves
$$E_{\bar\varpi^i}:y^2=x^3+\frac{\bar\varpi^{2i}}{4}.$$ 
and
$$E_{\varpi^i}:y^2=x^3+\frac{\varpi^{2i}}{4}.$$ 
Since $\bar\varpi$ and $\varpi$ are not distinct from each other, so if one statement holds for $E_{\bar\varpi}$, then it automatically holds for $E_{\varpi}$ if we change $\bar\varpi$ to $\varpi$ in the statement. The conversation is also true. So for simplicity, most of the time, we only state facts for one of $E_{\bar\varpi}$ and $E_{\varpi}$  in the following. 

The Tate algorithm shows that the conductor of $E_{\varpi^i}$ is the ideal 
\begin{equation}\label{conductor}
N(E_{\varpi^i})=\begin{cases} (\sqrt{-3})^2(\varpi)^2,&\text{if}\ \ p^i\equiv 4\mod 9;\\   (\sqrt{-3})^4(\varpi)^2,& \text{if}\ \ p^i\equiv 7\mod 9. \end{cases}
\end{equation}
Let $\psi$, resp. $\psi^c$ be the Hecke character of $E_{\varpi^i}$, resp. $E_{\bar\varpi^i}$. Then 
$$\psi(\fq)=\ov{\lrb{\frac{\varpi^i}{\varpi_\fq}}_3}\varpi_\fq$$ with $\fq\nmid 3\varpi$ and $\varpi_\fq\equiv 1\mod 3$. 
Let $\fm$ be the conductor of $\psi$, then \cite[P21,8.2.7]{Grossbook} tells us $\fm^2=N(E_{\varpi^i})$, i.e.
\begin{equation}
\fm=\begin{cases} (\sqrt{-3})(\varpi),&\text{if}\ \ p^i\equiv 4\mod 9;\\   (\sqrt{-3})^2(\varpi),& \text{if}\ \ p^i\equiv 7\mod 9. \end{cases}
\end{equation}
By \cite{Shimura71} (see also \cite{NM1}), the functions
\begin{equation}\label{def}
f(\tau)=\sum_{\substack{\fa\in I_K(\fm)\\ \fa:\mathrm{integral}}} \psi(\fa)e^{2\pi i N(\fa)\tau}
\end{equation}
and
\begin{equation}
f^c(\tau)=\sum_{\substack{\fa\in I_K(\bar\fm)\\ \fa:\mathrm{integral}}} \psi^c(\fa)e^{2\pi i N(\fa)\tau}
\end{equation}
are modular forms on $\Gamma_0(N)$ with Nebentypus character $\xi(d)=\lrb{\frac{-3}{d}}\frac{\psi(d)}{d}$ and $\xi^c(d)=\lrb{\frac{-3}{d}}\frac{\psi^c(d)}{d}=\bar{\xi}(d)$, where 
$$N=3\Norm(\fm)=\begin{cases} 9p,&\text{if}\ \ p^i\equiv 4\mod 9;\\   27p,& \text{if}\ \ p^i\equiv 7\mod 9. \end{cases}$$ 
and $I_K(\fm)$ is the group of fractional ideals in $K$ prime to $\fm$. Note $f$ is constructed from the Hecke character of $E_{\varpi}$ while $f^c$ from $E_{\bar\varpi}$. It can be easily checked that $\xi(d)=1$ if and only if $(d,N)=1$ and $d$ is a cube modulo $p$. So $f$ is a modular form for the congruence subgroup
\begin{equation}\label{cg}
\Gamma=\left\{\matrixx{a}{b}{c}{d} \in\Gamma_0(N)\mid d\equiv e^3\mod p\ \text{for some}\ e\in\BZ\right\}.
\end{equation}
If we write $f(\tau)=\sum_{n\geq 1}a_n e^{2\pi i n \tau}$ then it can be checked that $f^c(\tau)=\sum_{n\geq 1}\bar{a}_n e^{2\pi i n \tau}$ although $\psi^c\neq\bar\psi$. Also from (\ref{def}), we can see that $a_n\neq 0$ only if $n\equiv 1\mod 3$.

By \cite{Shimura71} and \cite{Shimura73} (see also \cite{GL}), we can construct an abelian variety $A_{f}$ from $f$. By construction, $A_{f}=J_1(N)/I_{f}(J_1(N))$ where $I_{f}$ is the annihilator of $f$ in the Hecke algebra acting on $J_1(N)$. Shimura proved that $A_{f}$ is defined over $\BQ$, but it splits to $E_{\bar\varpi^i}\times E_{\varpi^i}$ over $K$ by the result in \cite{GL}. The pullback of $\Omega^1(A_{f})$ is the space spanned by $\set{f,f^c}$. But beyond expectation, by \cite[Theorem 1.1(iv)]{GL}, the pullback of the invariant differential of $E_{\bar\varpi^i}$ corresponds to $f$ (rather than $f^c$).     So by \cite[section 3]{Shimura73} and \cite[Theorem 1]{GL}, $E_{\bar\varpi^i}$ can be parameterized by $X_1(N)$ over $K$ through the integral of $f(z)$, i.e.
\[\varphi: \xymatrix{X_1(N)\ar[r]&J_1(N)\ar[r]&E_L=\BC/L\ar[r]& E_{\bvarpi}}\]
with
\begin{equation}\label{modularpara}
\varphi: t\mapsto z_t=2\pi i\s\int_{i\infty}^t f(z)dz\mapsto \lrb{\wp_L(z_t), \frac{1}{2}\wp'_L(z_t)}\mapsto \lrb{\frac{1}{c^2}\wp_L(z_t), \frac{1}{2c^3}\wp'_L(z_t)}.
\end{equation}
Here 
\begin{equation}\label{L}
L=\set{2\pi i \int_{i\infty}^{\gamma \infty}f(\tau) d\tau\mid \gamma\in\Gamma_1(N)}
\end{equation}
is the period lattice of $f$ over $X_1(N)$ and $c\in K$ is the Manin-Steven constant such that $\varphi^* \omega_{E_{\bar\varpi^i}}=c\cdot 2\pi i f(\tau)d\tau$ where $\omega_{E_{\bar\varpi^i}}$ is the Neron differential of $E_{\bar\varpi^i}$. This parametrization factors through the modular curve $X_\Gamma$.  It is conjectured that the Manin-Steven constant $c\in\CO_K^\times$ \cite{GL01} which means $E_L=E_{\bar\varpi^i}$ (not only isomorphic over $K$).

\section{modular curves and Hecke actions}

Let $\fU$ be the open compact subgroup of $\GL_2(\widehat{\BZ})$ consisting of the matrix $\matrixx{a}{b}{c}{d}$ such that $c\equiv 0\mod N$ and $d$ is a cube modulo $p$. Let $X_\fU$ be the modular curves over $\BQ$ whose underlying Riemann
surfaces are
\[X_\fU(\BC)=\GL_2(\BQ)^+\left\backslash \left(\left(\CH\bigsqcup\BP^1(\BQ)\right )\times \GL_2(\BA_f)\right/\fU\right),\]

Let $\Gamma$ be the subgroup of $\Gamma_0(N)$ defined in \eqref{cg}. Note that $\GL_2(\BQ)^+\cap \fU=\Gamma$. So $X_\fU$ is the modular curve $X_\Gamma$.

Put
\[W=\matrixx{0}{1}{-N}{0},\quad A=\matrixx{1}{1/3}{0}{1},\]
and
\begin{equation*}
B=-\frac{1}{N}WA^{-1}W=\matrixx{1}{0}{N/3}{1}.
\end{equation*}
Let $\gamma=\matrixx{a}{b}{Nc}{d}$ be an element in $\Gamma$, then
\[W\gamma W^{-1}=\matrixx{d}{-c}{-Nb}{a}\in \Gamma,\]
\[A\gamma A^{-1}=\matrixx{Nc/3+a}{-Nc/9+b+(b-a)/3}{Nc}{-Nc/3+d}\in\Gamma.\]
As a result, the actions of $A, W$ and $B$ induce an isomorphism of the modular curve $X_{\Gamma}$. To investigate
the action of $A,B$, we need the following result on the $K$ rational points of $E_{\bvarpi}$.

\begin{proposition}\label{cusp}We have
$E_{\bvarpi}(K)=\set{O,(0,\pm \frac{\bvarpi}{2})}$. 
\end{proposition}
\begin{proof}
By Monsky's result \cite[Theorem 1.4]{Monsky}, $\bar\varpi^i$ is not a cube sum in $K$.
So, $E_{\bvarpi}(K)=E_{\bvarpi}(K)_{tor}$.

Let $\phi'$ be the twist map from $E_{\bvarpi}$ to $E_1$ given by $(x,y)\mapsto \lrb{\frac{x}{\sqrt[3]{{\bar\varpi}^{2i}}},\frac{y}{\bvarpi}}$. Then $\phi'(E_{\bvarpi}(K)_{\tor})\subset E_1(K(\sqrt[3]{\bar\varpi}))_{\tor}$. But one can prove that 
$$E_1(K(\sqrt[3]{\bar\varpi}))_{\tor}=E_1(K)=\set{O,\lrb{0,\pm 1/2},\lrb{-\omega^j,\pm \sqrt{-3}/2}}$$
with $j=0,1,2$ similar to \cite[Proposition 2.3]{SSY}. Then we can see 
$$E_{\bvarpi}(K)_{\tor}=E_{\bvarpi}[3](K)=\set{O,\lrb{0,\pm \frac{\bvarpi}{2}}}.$$ 
\end{proof}

\begin{proposition}\label{AB}
We have
\[f\mid_{A}=\omega f,\ \ \ f\mid_{B}=\omega^2 f.\]
For any $P\in X_{\Gamma}$
\begin{equation*}\label{BC}
 A\cdot (\varphi(P))=[\omega]\varphi(P),\ \ \ \ \ B\cdot(\varphi(P))=[\omega^2]\varphi(P).
\end{equation*}
The same formula also holds for $f^c$ and $\varphi^c$.
\end{proposition}
\begin{proof}
By \cite[Section 12.3]{Iwanbook}, $f\mid_{W}$ is a multiplicity of $f^c$. From (\ref{def}), we can see that $a_n\neq 0$ if and only if $n\equiv 1\mod 3$. So we know that $f\mid_A=\omega f$ and $f^c\mid_A=\omega f^c$. Then the action of $W$ and $A$ stabilize $I_{f}(J_{\Gamma})$. As a result, $W$ and $A$ induces an automorphism of $A_f$. Clearly, $A$ acts as the complex multiplication $[\omega]$ on $A_f$ while $W$ interchanges $E_{\bar\varpi^i}$ and $E_{\varpi^i}$.

From the action of $W$ and $A$, we know that $B$ induces the automorphisms of $E_{\bar\varpi^i}$ and $E_{\varpi^i}$ as genus one curves. By \cite[Page 71]{Silvermanbook1}, every automorphism of $E_{\bar\varpi^i}$ as curves has the form $Z\mapsto aZ+b$ with $a\in\set{\pm 1,\pm\omega,\pm\omega^2}$ and $b\in E_{\bar\varpi^i}$. We see that for any $P\in X_{\Gamma}$
\begin{equation*}
A\cdot (\varphi(P))=[\omega]\varphi(P),\ \ \ \ \ B\cdot(\varphi(P))=[\omega^2]\varphi(P)+\varphi([3/N])
\end{equation*}
where $\varphi([3/N])$ is obtained by setting $P=[\infty]$. But if let $P=[0]$, we get 
$$\varphi([3/N])=\varphi([0])-[\omega^2]\varphi([0])=O$$ 
by Proposition \ref{cusp}. Thus we have
\begin{equation*}
 A\cdot (\varphi(P))=[\omega]\varphi(P),\ \ \ \ \ B\cdot(\varphi(P))=[\omega^2]\varphi(P)
\end{equation*}
for any $P\in X_{\Gamma}$. 
\end{proof}

\section{The CM point}
Let $N$ be the level of the modular form $f$ as in section 2. Then
\[N=\begin{cases} 9p,& \text{if}\ \ p^i\equiv 4\mod 9,\\   27p,& \text{if}\ \ p^i\equiv 7\mod 9.\end{cases}\]
Let
\[M=\matrixx{0}{-1}{3}{3r},\]
where $r\in \BZ$. We choose the CM point to be $$\tau_r=M\omega=\frac{-1}{3\omega+3r}.$$ Then we have the normalized embedding $\iota_1$ of $K$ into the $2\times 2$ matrix algebra $M_2(\BQ)$ in the sense of \cite{Shimurabook} such that

\begin{equation}\label{embedding}
\iota_1(\omega)=M\matrixx{-1}{-1}{1}{0}M^{-1}=\matrixx{-r}{-1/3}{3(r^2-r+1)}{r-1},
\end{equation}
then
\[\iota_1(\omega^2)=\matrixx{r-1}{1/3}{-3(r^2-r+1)}{-r}.\]
The CM point $\tau_r$ is fixed by the fractional transformation fo $\iota_1(\omega)$. From now on,
we always set $r^2-r+1\equiv 0 \mod 3p$ (one can check this is not an empty condition since $p\equiv 1\mod 3$) and set
\[t=\frac{r^2-r+1}{3p}.\]
Since $r^2-r+1\equiv 0\mod 3$, we have $r\equiv 2\mod 3$, i.e., $r\equiv 2,5,8\mod 9$. We can check that in all cases, $t\equiv 1\mod 3$. Then we have the following decompositions of matrices.

If $N=9p$,
\begin{equation}\label{omega1}
A\iota_1(\omega)=\matrixx{-r+3pt}{\frac{(r-2)}3}{9pt}{r-1}\in\Gamma_0(N),
\end{equation}
\begin{equation}\label{omega0}
\iota_1(\omega)A=\matrixx{-r}{\frac{-(r+1)}3}{9pt}{3pt+r-1}\in\Gamma_0(N).
\end{equation}
If $N=27p$,
\begin{equation}\label{omega2}
B^2A\iota_1(\omega)=\matrixx{-r+3pt}{\frac{(r-2)}3}{\frac{(t-2r)N}3+2ptN}{r-1+\frac{2(r-2)N}9}\in\Gamma_0(N).
\end{equation}
\begin{equation}\label{omega3}
\iota_1(\omega)AB^2=\matrixx{\frac{-2(r+1)N}9-r}{\frac{-(r+1)}3}{9p(2r^2+t)}{r^2}\in\Gamma_0(N).
\end{equation}

We also consider the CM point $$W(\tau_r)=\frac{3\omega+3r}{N}$$ where $W$ is the Atkin-Lehner involution. Then $W(\tau_r)$ is fixed by the matrix
\begin{equation}\label{embedding1}
\iota_2(\omega)=\matrixx{r-1}{-\frac{3(r^2-r+1)}{N}}{\frac{N}{3}}{-r}.
\end{equation}
Similarly, we have the following decompositions.

 If $N=9p$, 
 \begin{equation}\label{omega02}
B^2\iota_2(\omega)=\matrixx{r-1}{-\frac{3(r^2-r+1)}{N}}{\frac{(2r-1)N}{3}}{-2r^2+r-2}\in\Gamma_0(N).
\end{equation}

\begin{equation}\label{omega12}
\iota_2(\omega)B^2=\matrixx{2r^2+3r-3}{-\frac{3(r^2-r+1)}{N}}{\frac{(1-2r)N}{3}}{-r}\in\Gamma_0(N).
\end{equation}

If $N=27p$, 
\begin{equation}\label{omega22}
B^2A\iota_2(\omega)=\matrixx{\frac{N}9+r-1}{\frac{-t-r}3}{\frac{(2r-1)N}3+\frac{2N^2}{27}}{\frac{-2rN}9 - 2r^2 + r - 2}\in\Gamma_0(N),
\end{equation}
\begin{equation}\label{omega32}
\iota_2(\omega)AB^2=\matrixx{\frac{2rN}{9} - 2r^2 - \frac{2N}{9} + 3r - 3}{\frac{(-t+r-1)}{3}}{\frac{(1-2r)N}{3}+\frac{2N^2}{27}}{\frac{N}{9}-r}\in\Gamma_0(N).
\end{equation}

The embedding $\iota_1$ can be extended to an embedding $\iota_1'$ of $K\otimes \BZ_p$ into $M_2(\BQ_p)$ such that 
$$\iota_1((a+b\omega)\otimes c)=\lrb{\matrixx{a}{0}{0}{a}+\matrixx{b}{0}{0}{b}\iota_1(\omega)}\matrixx{c}{0}{0}{c},$$
where $a,b\in\BQ$ and $c\in\BZ_p$. We can also extend $\iota_1$ to an embedding $\iota_1''$ of $K\otimes \widehat{\BZ}$ into $M_2(\widehat{\BQ})$ in the same way. The composition of $\iota_1''$ with the inverse of the isomorphism $\iota:K\otimes\widehat{\BZ}\ra \widehat{K}=\prod'_{\fq\leq\infty}K_\fq$ gives the embedding $\widehat{K}\hookrightarrow M_2(\widehat{\BQ})$.
We will chase the following commutative diagram:
\begin{equation}\label{embedding}
\xymatrix{\widehat{K}\ar[r]^{\iota^{-1}}& K\otimes\widehat{\BZ}\ar@{^{(}->}^{\iota_1''}[r]& M_2(\widehat{\BQ}) \\ K_\varpi\oplus K_{\bar\varpi}\ar[r]^{\iota_p^{-1}}\ar@{^{(}->}[u]& K\otimes_{\BZ}\BZ_p\ar@{^{(}->}^{\iota_1'}[r]\ar@{^{(}->}[u]&M_2(\BQ_p)\ar@{^{(}->}[u]\\ &K\ar@{^{(}->}[u]\ar@{^{(}->}[r]^{\iota_1}&M_2(\BQ)\ar@{^{(}->}[u]}
\end{equation}
Obviously, the above is also true for $\iota_2$.

Recall $r\in\BZ$ is a solution of $r^2-r+1\equiv 0\mod 3p$. Write $\varpi=u+v\omega$, there is a unique cubic root of unity in $\BZ_p$ which we denote $\alpha$ such that  $p|(u+v\alpha)$ in $\BZ_p$. The condition $-r\equiv \omega^2\mod\varpi$ for $r$ is equivalent to $-r\equiv \alpha^2\mod p$ in $\BZ_p$ while $-r\equiv \omega\mod\varpi$ for $r$ is equivalent to $-r\equiv \alpha\mod p$ in $\BZ_p$. We will also write $\tau$ for $[\tau,1]$ on the adelic modular curve $X_\fU$.

\begin{theorem}\label{cmpoint}
Let $r\in\BZ$ such that $r^2-r+1\equiv 0\mod 3p$. If $-r\equiv \omega^2\mod\varpi$, 
then $\varphi(\tau_r)$ is defined over $K(\sqrt[3]{\varpi})$ and 
\[\sigma_{\omega_\fp}(\varphi(\tau_r))=\begin{cases}[\omega^2]\varphi(\tau_r),& \text{if}\ \ N=9p;\\ [\omega]\varphi(\tau_r),& \text{if}\ \ N=27p. \end{cases}\]
If $-r\equiv \omega\mod\varpi$, then $\varphi(\tau_r)$ is defined over $K$ so is a $\sqrt{-3}$ torsion point.
\end{theorem}

\begin{proof}
The isomorphism $\iota_p: K\otimes_{\BZ}\BZ_p\cong K_\varpi\oplus K_{\bar\varpi}$ which is compatible with the diagonal embedding $K\hookrightarrow \widehat{K}^\times$ is given by $\omega\otimes 1\mapsto (\omega,\omega)$ and $1\otimes (u+v\alpha) \mapsto (\varpi,\bar\varpi)$ since $p$ are both uniformizers in $K_{\varpi}$ and $K_{\bar\varpi}$. As a result, $\iota_p(1\otimes\alpha)=(\omega,\bar\omega)$ and $\iota_p^{-1}((1,\omega))=(\omega^2\otimes 1)\cdot(1\otimes \alpha)$, $\iota_p^{-1}((\omega,1))=(\omega\otimes 1)\cdot(1\otimes \alpha)$. So 

\[\iota''_1\circ\iota^{-1}(\omega_{\bar\fp})= \alpha\iota'_1(\omega^2\otimes 1)=\matrixx{(r-1)\alpha}{\alpha/3}{-3(r^2-r+1)\alpha}{-r\alpha},\]

\[\iota''_1\circ\iota^{-1}(\omega_{\fp})= \alpha\iota'_1(\omega\otimes 1)=\matrixx{(-r)\alpha}{-\alpha/3}{3(r^2-r+1)\alpha}{(r-1)\alpha}.\]

If $-r\equiv \omega^2\mod\varpi$, then $-r\alpha\equiv 1\mod p$, we know that $\iota''_1\circ\iota^{-1}(\omega_{\bar\fp})\in \fU_{p}$ while $\iota''_1\circ\iota^{-1}(\omega_{\fp})\notin \fU_{p}$. 
Then the compact subgroup of $\widehat{K}$
\[U=(\BZ_3+3\CO_{K,3})^\times(\CO_{K,\varpi}^\times)^3\prod_{v\nmid 3\varpi}\CO_{K,v}^\times \]
satisfies $\iota''_1\iota^{-1}(U)\subset\fU$ where $(\CO_{K,\varpi}^\times)^3$ means cube elements in $\CO_{K,\varpi}^\times$. By the Shimura reciprocity law, $[\tau_r,1]$ is defined over the abelian extension of $K$ with Galois group $\sigma(K\widehat{\CO}_K^\times/KU)=\langle \sigma_{\omega_\fp}\rangle$ which is exactly $K(\sqrt[3]{\varpi})$ by Proposition \ref{LCF}. 

By the Hilbert reciprocity law, the actions of $\sigma_{\omega_3}$ and $\sigma_{\omega_{\fp}}$ on $\tau$ are inverse to each other, so we can study the action of $\sigma_{\omega_{\fp}}$ through $\sigma_{\omega_3}$.
By \eqref{omega1} and \eqref{omega2}, if $N=9p$,
\[A\iota_1(\omega)_3=\lrb{(A\iota_1(\omega))_3,A}\in \fU;\]
if $N=27p$,
\[B^2A\iota_1(\omega)_3=\lrb{(B^2A\iota_1(\omega))_3,B^2A}\in \fU.\]
By the Shimura reciprocity law,  
\[\sigma_{\omega_\fp}(\varphi(\tau))=\sigma_{\omega_3}^{-1}(\varphi(\tau))
=\begin{cases}A^{-1}\cdot\varphi(\tau),& \text{if}\ N=9p;\\ (B^2A)^{-1}\cdot\varphi(\tau),& \text{if}\ N=27p. \end{cases}\]
Now, the result follows from Proposition \ref{AB}.

If $-r\equiv \omega\mod\varpi$, then $(r-1)\alpha\equiv 1\mod p$, we know that $\iota''_1\circ\iota^{-1}(\omega_{\bar\fp})\notin \fU_{p}$ while $\iota''_1\circ\iota^{-1}(\omega_{\fp})\in \fU_{p}$. Similarly, $\varphi(\tau)$ is defined over $K(\sqrt[3]{\bar\varpi})$ and the Galois action is given by
\[\sigma_{\omega_{\bar\fp}}(\varphi(\tau))=\sigma_{\omega_3}^{-1}(\varphi(\tau))
=\begin{cases}A^{-1}\cdot\varphi(\tau),& \text{if}\ N=9p;\\ (B^2A)^{-1}\cdot\varphi(\tau),& \text{if}\ N=27p. \end{cases}\]
So $\varphi(\tau_r)$ is twisted to a point on $E_{\bar\varpi^{2i}}$ defined over $K$ under the map $(x,y)\mapsto \lrb{{\sqrt[3]{{\bar\varpi}^{2i}}}x,{\bvarpi}y}$. By Proposition \ref{cusp}, $\varphi(\tau_r)$ is $\sqrt{-3}$ torsion and defined over $K$.
\end{proof}
Similar proof using \eqref{omega02} and \eqref{omega22}  gives the following theorem.
\begin{theorem}\label{cmpoint1}
Let $r\in\BZ$ such that $r^2-r+1\equiv 0\mod 3p$. If $-r\equiv \omega\mod\varpi$, 
then $\varphi(W(\tau_r))$ is defined over $K(\sqrt[3]{\varpi})$ and 
\[\sigma_{\omega_\fp}(\varphi(\tau_r))=\begin{cases}[\omega^2]\varphi(\tau_r),& \text{if}\ \ N=9p;\\ [\omega]\varphi(\tau_r),& \text{if}\ \ N=27p. \end{cases}\]
If $-r\equiv \omega^2\mod\varpi$, then $\varphi(W(\tau_r))$ is defined over $K$ so is a $\sqrt{-3}$ torsion point.
\end{theorem}

\section{Cubic root of the modular functions}

Assume the modular parametrization in \eqref{modularpara}
$$\varphi: X_1(N)\longrightarrow E_{\bar{\varpi}^i}:y^2=x^3+\bar{\varpi}^{2i}/4$$
is realized by the modular functions $x(z), y(z)$, that is $\varphi(z)=(x(z),y(z))$. Then the parametrization
$$\varphi^c: X_1(N)\longrightarrow E_{{\varpi}^i}:y^2=x^3+\varpi^{2i}/4$$ is realized by the modular functions $x^c(z), y^c(z)$, where $x^c(z)$ and $y^c(z)$ are obtained from $x(z)$ and $y(z)$ by taking the complex conjugate of the Fourier coefficients. The functions $x(z)$ and $x^c(z)$ are of level $\Gamma_1(N)$ while $y(z)$ and $y^c(z)$ are of level $\Gamma_0(N/3)$ (proved below). In this section we will investigate the property of some cubic roots of these modular functions. Using these properties, we will also prove that the images of some cusps and CM points are primitive $\sqrt{-3}$ torsion points on $E_{\bar{\varpi}^i}$ and $E_{{\varpi}^i}$. For this we first need to study some period lattices of $f$. Recall the period lattice $L$ is defined in \eqref{L}.

\begin{lemma}\label{cm}
The lattice $L$ has complex multiplication by $\CO_K$.
\end{lemma}
\begin{proof}
Choose $\gamma_0\in\Gamma_0(N)$ such that $\xi(\gamma_0)=\omega$. For any $\gamma\in\Gamma_1(N)$,
\[\omega \int_{i\infty}^{\gamma \infty}f(\tau)d\tau =\int_{i\infty}^{\gamma \infty}f(\tau)\mid_{\gamma_0}d\tau= \int_{\gamma_0\infty}^{\gamma_0\gamma \infty}f(\tau)d\tau=\int_{\gamma_0\infty}^{\gamma'\gamma_0 \infty}f(\tau)d\tau\]
for some $\gamma'\in \Gamma_1(N)$. But $$2\pi i\int_{s}^{\gamma's}f(\tau) d\tau $$ is independent of $s$ and belongs to $L$. So $L$ is stable under the multiplication by $\omega$. 
\end{proof}

We will consider another two lattices spanned by the periods of $f$ over bigger congruence subgroups.
$$L_\xi=\left\langle 2\pi i\int_{i\infty}^{\gamma \infty}f(\tau) d\tau\mid \gamma\in\Gamma\right\rangle,$$ 
$$L_0=\left\langle 2\pi i\int_{i\infty}^{\gamma \infty}f(\tau) d\tau\mid \gamma\in\Gamma_0(N/3)\right\rangle.$$
where $\left\langle\ \right\rangle$ means $\BZ$ linear spanning.

\begin{lemma}\label{L=L}
We have $L_{\xi}=L$.
\end{lemma}
\begin{proof}
Let $\gamma=\matrixx{a}{b}{c}{d}\in\Gamma$. Since $\gamma \infty$ is defined over $\BQ$, the point $\varphi(\gamma\infty)$ is a $K$ point of $E_{\bar\varpi^i}$. So $\varphi(\gamma\infty)$ is $\sqrt{-3}$-torsion by Proposition \ref{cusp}. As a result we know that $|L_\xi/L|$ is a divisor of $3$. Since $d$ is a cube mod $p$, $\gamma^{\frac{p-1}{3}}\in \Gamma_1(N)$. So
$$\frac{p-1}{3}2\pi i\int_{i\infty}^{\gamma \infty}f(\tau)d\tau=2\pi i\int_{i\infty}^{\gamma^{\frac{p-1}{3}} \infty}f(\tau) d\tau\in L.$$
Since $|L_\xi/L|=1$ or $3$ and $(\frac{p-1}{3},3)=1$, we know that 
$$2\pi i\int_{i\infty}^{\gamma \infty}f(\tau) d\tau\in L.$$
\end{proof}

\begin{lemma}\label{L=L0}
We have $L_0=L$.
\end{lemma}
\begin{proof}
Let $\Gamma_H$ be the subgroup of $\Gamma_0(N/3)$ generated by $\Gamma$ and the matrix $B=\matrixx{1}{0}{N/3}{1}$, then $[\Gamma_0(N/3):\Gamma_H]=3$ and $\Gamma_0(N/3)/\Gamma_H$ is generated by any matrix $\matrixx{a}{b}{N/3}{d}$ where $a$ is not a cube modulo $p$. 

By Proposition \ref{AB}, $$2\pi i\int_{i\infty}^{B\infty}f(\tau)d\tau=(1-\omega^2)2\pi i\int_{i\infty}^{0}f(\tau)d\tau\in L.$$
We also know that (Lemma \ref{L=L}) 
$$2\pi i\int_{i\infty}^{\gamma\infty}f(\tau)d\tau\in L,\ \text{for}\ \gamma\in\Gamma .$$
So we just need to prove that there exists an integer $a$ which is not a cube modulo $p$ and
\[2\pi i\int_{\infty}^{3a/N}f(\tau)d\tau\in L.\]

Since $p\equiv 1\mod 3$, the congruent equation $r^2-r+1\equiv 0\mod 3p$ always has integer solutions. We choose $r$ a solution such that $\xi(r)=\omega$ or $\omega^2$ depending on $N=9p$ or $27p$ (this condition is in fact a condition modulo $p$). This implies $r$ and $r-1$ are not  cubes modulo $p$.
Consider the matrix
\[D=\matrixx{r-1}{-\frac{3(r^2-r+1)}{N}}{\frac{N}{3}}{-r},\]
as in \eqref{embedding1} with fixed point $z_0=W(\tau_r)$. 
By Proposition \ref{AB}, $f(z)\mid_{B^2}=\omega f(z), f(z)\mid_{A}=\omega f(z)$. Then from \eqref{omega12} and \eqref{omega32}, we know that $f(z)\mid_D=f(z)$.  So,
\[\int_{z_0}^{D\infty}f(\tau)d\tau=\int_{D^{-1}z_0}^{i\infty} f(\tau)\mid_D d\tau=\int_{z_0}^{i\infty}f(\tau)d\tau.\]
Now
\[\int_{i\infty}^{\frac{3(r-1)}{N}}f(\tau)d\tau=\int_{i\infty}^{D\infty}f(\tau)d\tau=\int_{i\infty}^{z_0}f(\tau)d\tau+\int_{z_0}^{D\infty}f(\tau)d\tau=0.\]
This finishes the proof of the Lemma.
\end{proof}

\begin{corollary}\label{zero}
There exist $\gamma\in\Gamma_0(N)$ not in $\Gamma$ such that 
\[\int_{i\infty}^{\gamma\infty} f(\tau)d\tau=0.\]
\end{corollary}
\begin{proof}Let $\gamma_0\in\Gamma_0(N)$ not in $\Gamma$.
Since $L_0=L$ by Lemma \ref{L=L}, there exists $\gamma_1\in\Gamma$ such that
\[\int_{i\infty}^{\gamma_0\infty}f(\tau)d\tau=\int_{i\infty}^{\gamma_1\infty}f(\tau)d\tau\]
i.e.
\[\int_{\gamma_1\infty}^{\gamma_0\infty}f(\tau)d\tau=\int_{i\infty}^{\gamma_1^{-1}\gamma_0\infty}f(\tau)d\tau=0.\]
Now $\gamma=\gamma_1^{-1}\gamma_0$ satisfies the condition of the Corollary.
\end{proof}


All the above discussions obviously apply to $f^c$. Let 
$$L^c=\set{2\pi i \int_{i\infty}^{\gamma \infty}f^c(\tau) d\tau\mid \gamma\in\Gamma_1(N)}$$
then obviously the lattice $L^c$ is the complex conjugate of the lattice $L$. 

\begin{proposition}The functions
$y(z)$ and $y^c(z)$ are modular functions of $\Gamma_0(N/3)$.
\end{proposition}
\begin{proof}
We only give the proof for $y(z)$, the proof for $y^c(z)$ is similar.
Let $(\wp_L, \wp_L')$ be Weierstrass $\wp$-functions of $E_{L}$. Since $E_{L}$ has CM by $\CO_K$, we know that $\wp_L'(\omega x)=\wp_L'(x)$. Let $\gamma\in\Gamma_0(N/3)$, then
\begin{equation}\label{yyc}
\int_{i\infty}^{\gamma z}f(\tau)d\tau=\int_{i\infty}^z f\mid_{\gamma}(\tau)d\tau+\int_{i\infty}^{\gamma\infty}f(\tau)d\tau.
\end{equation}
Note that $\Gamma_0(N/3)$ is generated by $\Gamma_0(N)$ and the matrix $B$. By Proposition \ref{AB}, $f\mid_{\gamma}(\tau)=\omega^k f(\tau)$ for some $k=0,1,2$. By Lemma \ref{L=L}, $$2\pi i\int_{i\infty}^{\gamma\infty}f(\tau)d\tau\in L.$$ Finally by \eqref{yyc},
we have $$y(\gamma z)=\frac{1}{2c^3}\wp_L'\lrb{2\pi i\int_{i\infty}^{\gamma z} f(\tau)d\tau}=\frac{1}{2c^3}\wp_L'\lrb{2\pi i\int_{i\infty}^zf(\tau)d\tau}=y(z).$$
\end{proof}

Let $[w,\omega w]$ be a basis of $L$ and let $L_{3}=[3w, 2\omega w+w]=\sqrt{-3}L$ be the sub-lattice of index $3$ of $L$. The Weierstrass equation of $cL_3$ is given by $$C_{\bar\varpi}: y^2=x^3-\frac{\bar\varpi^{2i}}{4\cdot27}.$$ 
The primitive $3$-torsion points of  $C_{\bar\varpi}$ are $\set{\lrb{\omega^j\frac{\sqrt[3]{\bar\varpi^{2i}}}{3},\pm\frac{\bar\varpi^{i}}{6}}: j=0,1,2}$.  
Let 
\begin{equation}\label{WW}
K(z)=\frac{\frac{\wp'_{L_3}(z)}{2c^3}-\frac{\bar\varpi^{i}}{6}}{\frac{\wp_{L_3}(z)}{c^2}}, \ \ \ J(z)=\frac{\frac{\wp'_{L_3}(z)}{2c^3}+\frac{\bar\varpi^{i}}{6}}{\frac{\wp_{L_3}(z)}{c^2}}.
\end{equation}
By comparing the zeros and poles on $\BC/cL_3$, we know that 
$$K(z)^3=\frac{\wp'_L(z)}{2c^3}-\frac{\bar\varpi^i}{2},\ \ \ J(z)^3=\frac{\wp'_L(z)}{2c^3}+\frac{\bar\varpi^i}{2}.$$ 
Further let
\begin{equation}\label{WWW}
\ff(z)=K \lrb{2\pi i\int_{i\infty}^zf(\tau)d\tau}, \ \ \ \fg(z)=J \lrb{2\pi i\int_{i\infty}^zf(\tau)d\tau},
\end{equation}
then
$$\ff(z)^3=y(z)-\frac{\bar\varpi^i}{2},\ \ \ \fg(z)^3=y(z)+\frac{\bar\varpi^i}{2}.$$
We can similarly define $K^c(z), J^c(z), \ff^c(z)$ and $\fg^c(z)$.
 Now, define
 \begin{equation}\label{modular-function}
F_+(z)=\frac{\ff(z)}{\ff^c(z)},\ \ F_-(z)=\frac{\fg(z)}{\fg^c(z)},\ \ G(z)=\ff(z)\fg^c(z),\ \ H(z)=\ff^c(z)\fg(z).
\end{equation}
We will prove these functions are invariant under $\Gamma$ in Section 8. But for this we need a lot of preparations.

\begin{proposition}\label{mk}
We have $K(\omega^k z)=\omega^{-k}K(z)$. The same is also true for $K^c(z), J(z)$ and $J^c(z)$.
\end{proposition}
\begin{proof}
Since $K(\omega^k z)^3=K(z)^3$,  $K(\omega^k z)=\kappa(z)K(z)$ such that $\kappa(z)^3=1$.  But both $K(\omega^k z)$ and $K(z)$ are continuous functions, so is $\kappa(z)$ which means $\kappa(z)$ is a constant. This proves $K(\omega^k z)/K(z)$ is a cubic root of unity.
Since the leading term of $K(z)$ is $z^{-1}$, we get the relation in the Proposition. 
\end{proof}

\begin{proposition}\label{LK}
The period lattice $L$ acts on $K(z)$ through a nontrivial cubic  character $\nu$, i.e. for any $w\in L$,
\[K(z+w)=\nu(w)K(z)\]
and $\nu(w)=\nu(\omega w)$. The same is also true for $K^c(z)$, $J(z)$ and $J^c(z)$.
\end{proposition}
\begin{proof}The proof that $w\in L$ acts on $K(z)$ by a cubic root is similar to the proof of Proposition \ref{mk}. We now prove $\nu$ is a character. For $w_1,w_2\in L$, $\nu(w_1+w_2)=(K(z+w_1+w_2)/K(z+w_1))\cdot(K( z+w_1)/K(z))=\nu(w_2)\nu(w_1)$.

 In the following we prove that $\nu(w)=\nu(\omega w)$. Since the $L$ invariant functions are generated by $\wp_L(z)$ and $\wp_L'(z)$ see \cite[Chapter VI]{Silvermanbook1}, $K(z)$ cannot be $L$ invariant. 

Let $\gamma\in\Gamma$ such that 
$$\alpha=2\pi i\int_{i\infty}^{\gamma \infty}f(\tau)d\tau,$$ and $\set{\alpha,\omega\alpha}$ is a basis of $L$.  By Corollary \ref{zero}, there exists $\gamma'\in \Gamma_0(N)$ such that $\gamma_1=\gamma\gamma'\in\Gamma_0(N)$ not in $\Gamma$ and
$$2\pi i\int_{i\infty}^{\gamma_1\infty}f(\tau)d\tau=\alpha$$
and without loss of generality assume $\xi(\gamma_1)=\omega$. We consider the action of $\gamma_1$ and $\gamma_1^2$ on $\ff$. Since

\[\int_{i\infty}^{\gamma_1^2 \infty}f(\tau)d\tau=(\xi(\gamma_1)+1)\int_{i\infty}^{\gamma_1\infty}f(\tau)d\tau=(\xi(\gamma_1)+1)\alpha,\]
we get
\begin{eqnarray*}
K\lrb{2\pi i\int_{i\infty}^{\gamma_1^2 z}f(\tau)d\tau}&=&\nu((\xi(\gamma_1)+1)\alpha)\cdot K\lrb{2\pi i\xi(\gamma_1)^2\int_{i\infty}^{z}f(\tau)d\tau}\\
&=&\nu((\xi(\gamma_1)+1)\alpha)\cdot \xi(\gamma_1)^{-2}K\lrb{2\pi i\int_{i\infty}^{z}f(\tau)d\tau}\\
\end{eqnarray*}
Similarly,
\[K\lrb{2\pi i\int_{i\infty}^{\gamma_1 z}f(\tau)d\tau}=\nu(\alpha)\cdot \xi(\gamma_1)^{-1} K\lrb{2\pi i\int_{i\infty}^{ z}f(\tau)d\tau}.\]
As a result, $\nu((\xi(\gamma_1)+1)\alpha)\cdot \xi(\gamma_1)^{-2}=(\nu(\alpha)\cdot \xi(\gamma_1)^{-1})^2$ which implies 
$$\nu(\alpha)\nu(\xi(\gamma_1)\alpha)=\nu(\alpha)^2, \text{i.e.}\ \nu(\omega\alpha)=\nu(\alpha).$$
But $(\alpha,\omega\alpha)$ is a basis of $L$, we finished the proof for $K(z)$. The proof for other functions is similar. 
\end{proof}

\begin{proposition}\label{torsion0}
We have $\varphi(0)$ resp. $\varphi^c(0)$ is a primitive $\sqrt{-3}$ torsion point on $E_{\bar\varpi^i}$ resp. $E_{\varpi^i}$.
\end{proposition}
\begin{proof}
Since $\varphi(0)$ is defined over $K$, we know that $\varphi(0)$ is a $\sqrt{-3}$ torsion point on $E_{\bar\varpi^i}$. We assume $\varphi(0)=O$, then 
\[2\pi i\int_{i\infty}^{0}f(\tau)d\tau\in L\]
and the modular function $\ff(z)$ has a pole at $z=0$. The $q$-expansion of $\ff(z)$ at the cusp $0$ is the $q$-expansion of $\ff(z)\mid_W$ at the cusp $\infty$. The proof of \cite[Theorem 12.5]{Iwanbook} shows 
$f|_{W}=Cf^c$ for some constant $C$.
By the equation 
\begin{eqnarray*}
\ff(z)\mid_W&=&K\lrb{2\pi i\int_{i\infty}^{W(z)}f(\tau)d\tau}\\
&=&\nu\lrb{2\pi i\int_{i\infty}^{0}f(\tau)d\tau }K\lrb{2\pi i\int_{0}^{W(z)}f(\tau)d\tau}\\
&=&\nu\lrb{2\pi i\int_{i\infty}^{0}f(\tau)d\tau }K\lrb{2\pi i C\int_{i\infty}^{z}f^c(\tau)d\tau}\\
&=&\nu\lrb{2\pi i\int_{i\infty}^{0}f(\tau)d\tau }K\lrb{ C\sum_{n=1}^\infty \frac{\bar{a}_n}{n}q^n}
\end{eqnarray*}
we know that the leading term of the $q$-expansion of $\ff(z)$ at $0$ is $(Cq)^{-1}$ up to a cubic root of unity.
As a result, the modular function $\ff(z)j(z)^{-1}$ is holomorphic and nonzero at $z=0$. Since
\begin{eqnarray*}
K\lrb{2\pi i\int_{i\infty}^{Bz}f(\tau)d\tau}&=&K\lrb{2\pi i\int_{i\infty}^{B\infty}f(\tau)d\tau +2\pi i\int_{B\infty}^{Bz}f(\tau)d\tau}\\
&=&\nu\lrb{2\pi i\int_{i\infty}^{B\infty}f(\tau)d\tau }\omega^{-2} K\lrb{2\pi i\int_{i\infty}^{z}f(\tau)d\tau},
\end{eqnarray*}
we know that 
$$\ff(Bz)j(Bz)^{-1}=\nu\lrb{2\pi i\int_{i\infty}^{B\infty}f(\tau)d\tau }\omega^{-2}\ff(z)j(z)^{-1}.$$
Letting $z=0$, we get
\begin{equation}\label{action0}
\nu\lrb{2\pi i\int_{i\infty}^{B\infty}f(\tau)d\tau }=\omega^2.
\end{equation}
So,
\[2\pi i\int_{i\infty}^{B\infty}f(\tau)d\tau\notin \sqrt{-3}L.\]
But this contradicts the assumption that $\varphi(0)=O$, because 
\[2\pi i\int_{i\infty}^{B\infty}f(\tau)d\tau=(1-\omega^2)2\pi i\int_{i\infty}^{0}f(\tau)d\tau.\]
\end{proof}

\section{The Manin-Stevens constant}\label{c}

Recall that we have the optimal parametrization $\varphi: X_1(N)\rightarrow E_{\bar\varpi^i}$ defined over $K$ and $\omega_{E_{\bar\varpi^i}}$ is the Neron differential. Since the action of Hecke operators is compatible with the map $\varphi$, $\varphi^*\omega_{E_{\bar\varpi^i}}$ is an eigenvector for the Hecke algebra, with the same eigenvalues as $f(q)\frac{dq}{q}$. By multiplicity one principle $\varphi^*\omega_{E_{\bar\varpi^i}}=c\cdot f(q)\frac{dq}{q}$ with $c\in K$. 
The ideal $(c)$ is the Manin ideal defined in \cite{GL01}. It is easy to see that the Manin-Stevens constant of $E_{\varpi^i}$ is the complex conjugation $\bar c$. Let $E_{p^{2i}}: y^2=x^3+\frac{p^{6-2i}}{4}$ and $c'\in\BZ$ be the Manin-Stevens constant of $E_{p^{2i}}$.
In this section we will prove that 
\begin{thm}\label{M-S}
The Manin-Stevens constant $c$ of $E_{\bar\varpi^i}$ is a root of unity in $\CO_K$. The Manin-Stevens constant $c'$ of $E_{p^{2i}}$ is $\pm 1$. 
\end{thm}
We will see that $c'=\pm 1$ is a byproduct of the proof that $c\in\CO_K^\times$ in the end of this section.     By \cite[Proposition 4.1 and 4.2]{GL01}, $c\in\CO_K$ is coprime to $2N$. To prove Theorem \ref{M-S}, it is enough to prove that $c$ is coprime to $2,3,p$. 

\subsection{Coprime to $2$}

 Let $\widetilde\varphi: J_1(N)\rightarrow E_{\bar\varpi^i}$ be the optimal quotient defined over $K$. Let $\CE_{\bar\varpi^i}$ be the Neron model of $E_{\bar\varpi^i}$ over $\CO_K$ and $\CJ_{\CO_K}$ be the Neron model of $J_1(N)$ over $\CO_K$. Then the map $\varphi$ extends to $\CJ_{\CO_K}\rightarrow \CE_{\bar\varpi^i}$ by the Neron mapping property.
The arguments in \cite[section 2.2, 2.4]{Ce}(Note that this paper is the first version which is different from the published version \cite{Ce1}) go through if we consider $\CE_{\bar\varpi^i}$, replacing $\BQ$ with $K$, $\BZ$ with $\CO_K$ and $\CJ_\BZ$ with $\CJ_{\CO_K}$ therein. Then similar results as \cite[Theorem 2.5]{Ce} show that the Manin-Stevens constant $c$ is coprime to $2$.

\subsection{Coprime to $3$}By Proposition \ref{torsion0}, the image of the cusp $0$ is a primitive $\sqrt{-3}$ torsion point under the modular parametrization. So we know that 
$$L(f,1)=2\pi i\int_{i\infty}^0 f(\tau)d\tau\neq 0.$$ 
and $\frac{L(f,1)}{\Omega}=\frac{m}{\sqrt{-3}}$ with $m\in\CO_K$ coprime to $\sqrt{-3}$ where $\Omega$ is a $\CO_K$ base of the period lattice of $f$.
By Theorem 1 and Corollary 1 in \cite{BF}, we have
\[\frac{L(E_{\bar\varpi^i} / K, 1)}{\Omega(E_{\bar\varpi^i})}=\frac{|\Sha(E_{\bar\varpi^i} / K)|}{|E_{\bar\varpi^i}(K)|^2} \cdot \prod_vc_v,\]
and
\[\frac{L(f, 1){\ov{L(f, 1)}}}{\Omega\bar\Omega c\bar{c}}=\frac{|\Sha(E_{\bar\varpi^i} / K)|}{|E_{\bar\varpi^i}(K)|^2} \cdot \prod_vc_v\]
where $c$ is the Manin-Steven constant and $c_v$ is the Tamagawa number at different places.  Since $|E_{\bar\varpi^i}(K)|=3$ and $c_{\bar\varpi}=3$ and $c_{\sqrt{-3}}=1,2,4$ and $c_v=1$ for $v\neq \bar\varpi, \sqrt{-3}$. We see that $c$ is coprime to $\sqrt{-3}$ and the primary 3-part of $|\Sha(E_{\bar\varpi^i} / K)|$ is trivial.

\subsection{Coprime to $p$}
Let 
\begin{equation}\label{root}
W(\psi)=\frac{\sqrt{-1}\psi_\infty(\fm)}{N(\fm)^{1/2}}\sum_{a\in\CO_K/\fm}\ov{\lrb{\frac{\varpi^i}{\fa}}}_3e\lrb{\text{Tr}\frac{a}{b}}
\end{equation}
be the Gauss sum of $\psi$ as defined in \cite[(12.21)]{Iwanbook}(see also \cite[(3.3.8)]{Miyake}) where $b$ is a generator of the ideal $\sqrt{-3}\fm$.  By \cite[(12.22)]{Iwanbook}, $|W(\psi)|^2=N(\fm)$. Let $C=-W(\psi)N(\fm)^{-1/2}$, the proof of \cite[Theorem 12.5]{Iwanbook} shows 
\begin{equation}\label{AL}
f|_{W}=Cf^c. 
\end{equation}
For any $\gamma\in\Gamma_1(N)$, there exist $\gamma'\in\Gamma_1(N)$ such that $\gamma W=W\gamma'$. Then we have
\[\int^{\gamma\infty}_{\infty}f(\tau)d\tau=\int^{\gamma W\infty}_{W\infty}f(\tau)d\tau=\int^{W\gamma'\infty}_{W\infty}f(\tau)d\tau=\beta\frac{\varpi^{\frac{i}{3}}}{\bar\varpi^{\frac{i}{3}}}\int^{\gamma'\infty}_{\infty}f^c(\tau)d\tau.\]
So the period lattice 
\begin{equation}\label{ffc}
L_f=CL_{f^c},
\end{equation}
where
$$L_f=\set{2\pi i \int_{i\infty}^{\gamma \infty}f(\tau) d\tau\mid \gamma\in\Gamma_1(N)},$$
$$L_{f^c}=\set{2\pi i \int_{i\infty}^{\gamma \infty}f^c(\tau) d\tau\mid \gamma\in\Gamma_1(N)}.$$
We use the notation $L_f,L_{f^c}$ in place of $L,L^c$ here in order to compare them with the period lattice of other modular forms below. 

Let $g(\tau)$ be the newform of the elliptic curve $E_{p^{2i}}: y^2=x^3+\frac{p^{6-2i}}{4}$. As before, assume $f(\tau)$ has the $q$-expansion $f(\tau)=\sum_{n=1}^\infty a_n q^n$ at infinity and $g(\tau)$ has the $q$-expansion $g(\tau)=\sum_{n=1}^\infty b_n q^n$ at infinity. Then $b_n\in\BZ$ for any $n$. We firstly show that $g$ is the twist of $f$ by a cubic Dirichlet character of modulus $p$. 
\begin{proposition}\label{twist}
Let $\bar\chi(n)={\lrb{\frac{\varpi^i}{\ell}}_3}$ and $\chi(n)={\lrb{\frac{\bar\varpi^i}{\ell}}_3}=\bar\chi(n)^{-1}$ be the primitive character of conductor $p$. 
The newform $g$ is the cubic twist of $f$ (resp. $f^c$) by the Dirichlet character $\ov{\chi}$ (resp. $\chi$), that is 
\[g(\tau)=\sum_{n=1}^{\infty}{\bar\chi}(n)a_nq^n.\]
resp.
\[g(\tau)=\sum_{n=1}^{\infty}\chi(n)\bar{a}_nq^n.\]
\end{proposition}
\begin{proof}

Since both $f$ and $g$ are eigenforms for all Hecke operators (including the bad primes where the Hecke operators is in fact the $U_p$-operators), it is enough to prove that $b_\ell=\chi(\ell) a_\ell$ for all prime $\ell$.
Since $a_\ell=b_\ell=0$ for $\ell=3$, $\ell\equiv 2\mod 3$ and $b_p=0$, we just consider the primes $\ell$ such that $\ell\equiv 1\mod 3$ and $\ell\neq p$. In this case, let $\ell=\pi\bar\pi$ with $\pi\equiv 1\mod 3$. Then
$$a_\ell=\ov{\lrb{\frac{\bar\varpi^i}{\pi}}_3}\pi+\ov{\lrb{\frac{\bar\varpi^i}{\ov\pi}}_3}\ov\pi,$$
$$b_\ell=\ov{\lrb{\frac{p^{2i}}{\pi}}_3}\pi+\ov{\lrb{\frac{p^{2i}}{\ov\pi}}_3}\ov\pi.$$
Since 
\[{\lrb{\frac{p^i\varpi^i}{\bar\pi}}_3}^{-1}={\lrb{\frac{p^i\bar\varpi^i}{\pi}}_3}\ \text{and}\ {\lrb{\frac{p^i\bar\varpi^i}{\pi}}_3}{\lrb{\frac{p^i\varpi^i}{\pi}}_3}=\lrb{\frac{p^{3i}}{\pi}}_3=1,\]
we see that
\[{\lrb{\frac{p^i\varpi^i}{\bar\pi}}_3}={\lrb{\frac{p^i\varpi^i}{\pi}}_3}.\]
Since 
\[{\lrb{\frac{p^i\varpi^i}{\bar\pi}}_3}{\lrb{\frac{p^i\varpi^i}{\pi}}_3}={\lrb{\frac{p^i\varpi^i}{\ell}}_3},\]
we get
\[{\lrb{\frac{p^i\varpi^i}{\bar\pi}}_3}={\lrb{\frac{p^i\varpi^i}{\pi}}_3}=\ov{\lrb{\frac{p^i\varpi^i}{\ell}}_3}.\]
As a result we have
\[b_\ell={\lrb{\frac{p^i\varpi^i}{\ell}}_3} a_\ell \]
It can be easily check that ${\lrb{\frac{p^i}{\ell}}_3}=1$, so
\[b_\ell=\bar\chi(\ell)a_\ell.\]

The proof for $f^c$ is similar.
\end{proof}


\begin{lemma}\label{cusp1}
Let $b/d$ be a cusp such that $d\equiv 1\mod N$, then for any $j\in\BZ_{\geq 0}$ and $u\in\set{1,\cdots,p-1} $, we have $p^jb/d+u/p$ is $\Gamma_1(Np)$ equivalent to $u/p$.
\end{lemma}
\begin{proof}
Note that $\frac{p^jb}{d}+\frac{u}{p}=\frac{p^{j+1}b+ud}{dp}$. Since $d\equiv 1\mod N$ and $(Np,p)=p$, the result is deduced from \cite[Lemma 3.2]{Cremona}.
\end{proof}

Let $\tilde{f}(\tau)=f(\tau)-B_pU_pf(\tau)=f(\tau)-a_pf(p\tau)=\sum_{(n,p)=1}a_nq^n$, then $\tilde{f}(\tau)$ is a modular form on $\Gamma_0(Np)$ with Nebentypus $\xi$ by \cite[section 1]{Li75}. Inductively we can see that

\begin{equation}\label{fexpand}
f(\tau)=\tilde{f}(\tau)+a_pf(p\tau)=\sum_{k=0}^{k_0-1} a_p^k \tilde{f}(p^k\tau)+a_p^{k_0}f(p^{k_0}\tau)=\sum_{k=0}^{\infty} a_p^k \tilde{f}(p^k\tau).
\end{equation}

For $\gamma\in\Gamma$,
\[\int_{i\infty}^{\gamma\infty}f(\tau)d\tau=\int_{0}^{\gamma 0}f(\tau)d\tau\]
We will consider the latter integral in the following.

By Proposition \ref{twist}, $\tilde{f}(\tau)=g_{\chi}(\tau)$ is the cubic twist of $g$ by $\chi$. Then by \cite[section 3]{AL} or \cite[section 8]{MTT}, 
\[\tilde{f}(\tau)=\frac{1}{\tau(\bar{\chi})}\sum_{u=1}^{p-1}\bar\chi(u)g\lrb{\tau+\frac{u}{p}}=\frac{1}{\tau(\bar{\chi})}\sum_{u=1}^{p-1}\bar\chi(u)g\lrb{\tau+\frac{-u}{p}}\]
where $\tau(\bar\chi)=\sum_{u=0}^{p-1}\bar{\chi}(u)e^{2\pi i\frac{u}{p}}$ is the Gauss sum of $\bar{\chi}$. So for any $\gamma\in\Gamma_1(N)$, 
\begin{eqnarray*}
\int_{0}^{\gamma 0}\tilde{f}(p^k\tau)d\tau &=&\frac{1}{p^k}\int_{0}^{p^k\gamma 0}\tilde{f}(\tau)d\tau \\
&=&\frac{1}{p^k\tau(\bar\chi)}\sum_{u=1}^{p-1}\bar\chi(u)\int_{\frac{u}{p}}^{p^k\gamma 0+\frac{u}{p}}g(\tau)d\tau
\end{eqnarray*}
This together with \eqref{fexpand} gives
\begin{equation}\label{integral-twist}
\int_{0}^{\gamma 0}f(\tau)d\tau=\sum_{k=0}^{n-1}\frac{a_p^k}{p^k\tau(\bar\chi)}\sum_{u=1}^{p-1}\bar\chi(u)\int_{\frac{u}{p}}^{p^k\gamma 0+\frac{u}{p}}g(\tau)d\tau+\frac{a_p^{n}}{p^n}\int_{0}^{p^n\gamma 0}f(\tau)d\tau.
\end{equation}
We assume $n_0$ is the least positive integer such that $p^{n_0}\gamma 0-\gamma0\in\BZ$.  Then, since $a_p=\bar\varpi$, we see
\[\lrb{1-\frac{1}{\varpi^{n_0}}}\int_{0}^{\gamma 0}f(\tau)d\tau=\sum_{k=0}^{n_0-1}\frac{1}{\varpi^k\tau(\bar\chi)}\sum_{u=1}^{p-1}\bar\chi(u)\int_{\frac{u}{p}}^{p^k\gamma 0+\frac{u}{p}}g(\tau)d\tau\]
Since $\chi(-1)=1$, $\tau(\bar\chi)=\ov{\tau(\chi)}$. It is well-known that $\tau(\chi)\tau(\bar\chi)=\tau(\chi)\ov{\tau(\chi)}=p$, then we have
\[\int_{0}^{\gamma 0}f(\tau)d\tau=\frac{\tau(\chi)}{\bar\varpi(\varpi^{n_0}-1)}\sum_{k=0}^{n_0-1}{\varpi^{n_0-1-k}}\sum_{u=1}^{p-1}\bar\chi(u)\int_{\frac{u}{p}}^{p^k\gamma 0+\frac{u}{p}}g(\tau)d\tau\]
Since $\Gamma_1(N)$ is finitely generated, then by Lemma \eqref{cusp1}, the period lattices of $f$ and $g$ satisfy 
\begin{equation}\label{fcg}
L_f\subset \frac{1}{m'}\tau(\chi)L_g
\end{equation}
for some $m'\in\CO_K$ coprime to $\varpi$. Here 
\[L_g=\set{2\pi i \int_{i\infty}^{\gamma \infty}g(\tau) d\tau\mid \gamma\in\Gamma_1(Np)}.\]
Since $g=f_{\bar\chi}$ is the cubic twist of $f$ by $\bar\chi$, similar argument shows for any $\gamma'\in\Gamma_1(Np)$,
\begin{equation}\label{sqrt-3}
\int_{\infty}^{\gamma' \infty}g(\tau)d\tau=\int_{0}^{\gamma' 0}g(\tau)d\tau=\frac{1}{\tau(\chi)}\sum_{v=1}^{p-1}\chi(v)\int_{\frac{v}{p}}^{\gamma'0+\frac{v}{p}}f(\tau)d\tau
\end{equation}
This means 
\begin{equation}\label{gcf}
\tau(\chi)L_g\subset L_f.
\end{equation}
\eqref{fcg} together with \eqref{gcf} implies
\begin{equation}\label{fg}
L_f=\frac{1}{m}\tau(\chi)L_g
\end{equation} 
for some $m\in \CO_K$ such that $m\mid m'$. 

By the general theory of Gauss sums and Jacobi sums \cite[Theorem 1 and its Corollary, Page 93]{IRbook}, $\tau(\chi)^3=pJ(\chi,\chi)$, $J(\chi,\chi)\in\CO_K$ and $J(\chi,\chi)\ov{J(\chi,\chi)}=p$. So $\tau(\chi)=\sqrt[3]{p\varpi}$ or $\sqrt[3]{p\bar\varpi}$ up to a sixth root of unity. Recall $c'\in\BZ$ is the Manin-Stevens constant of $E_{p^{2i}}$, then the Neron lattice $L_{E_{p^{2i}}}=c'L_g$ while $L_{E_{\bar\varpi^{i}}}=cL_f$. Together with \eqref{fg}, we get
\begin{equation}\label{LEE}
L_{E_{\bar\varpi^i}}=\frac{c}{c'}\cdot\frac{\tau(\chi)}{m}L_{E_{p^{2i}}}.
\end{equation}
By \cite[Theorem 1.2]{Ce1}, $c'$ can only have prime factors $3$ and $p$ while we have proved the Manin-Stevens constant $c$ of $E_{\bar\varpi^i}$ can only have primes factors above $p$. From the Weierstrass equations of $E_{\bar\varpi^i}$ and $E_{p^{2i}}$, we can see that (please consult \cite[Proposition VI.3.6]{Silvermanbook1} with \cite[example II.1.3.2]{Silvermanbook2})
$$\sqrt[3]{p\varpi}L_{E_{p^{2}}}=L_{E_{\bar\varpi}},\ \ \sqrt[3]{\frac{p}{\bar\varpi^2}}L_{E_{p^{4}}}=L_{E_{\bar\varpi^2}}.$$ 
So in \eqref{LEE} up to a sixth root of unity, $c'$ is coprime to $3$, $m=\bar\varpi^j$ for some integers $j\geq 0$ (we already know that $m$ is coprime to $\varpi$) and
\begin{equation}
\tau(\chi)=\begin{cases}
\sqrt[3]{p\varpi}& \text{if}\ i=1;\\ & \\ \sqrt[3]{p\bar\varpi}& \text{if}\ i=2. 
\end{cases}\end{equation}
i.e.
\begin{equation}
L_f=\begin{cases}
\frac{\sqrt[3]{p\varpi}}{\bar\varpi^j}L_g& \text{if}\ i=1;\\ &\\ \frac{\sqrt[3]{\frac{p}{\bar\varpi^2}}}{\bar\varpi^j}L_g& \text{if}\ i=2. 
\end{cases}\end{equation}
Since the complex conjugation $\ov{L_f}=L_{f^c}$ and $L_g=\ov{L_g}$, we can deduce (from \eqref{AL} and \eqref{ffc}) that
\begin{equation}\label{W}
C=\begin{cases}\frac{\varpi^{j+\frac{1}{3}}}{\bar\varpi^{j+\frac{1}{3}}}\  \text{with}\ j\geq 0 & \text{if}\ i=1; \\ &\\ \frac{\varpi^{j-\frac{1}{3}}}{\bar\varpi^{j-\frac{1}{3}}}\ \text{with}\ j>0 & \text{if}\ i=2,
\end{cases}\end{equation}
up to a sixth root of unity. From \eqref{root}, we see that $N(\fm)C$ is an algebraic integer. Combining this with \eqref{W}, we know that 
\begin{equation}\label{involution}
f\mid_{W}=\beta\frac{\varpi^{\frac{i}{3}}}{\bar\varpi^{\frac{i}{3}}}f^c, 
\end{equation}
with $\beta$ is a sixth root of unity. Again by \eqref{ffc}, the period lattice $$L_f=\frac{\varpi^{\frac{i}{3}}}{\bar\varpi^{\frac{i}{3}}} L_{f^c}.$$ 
Assume the Weierstrass equation of $E_{L_f}$ is $y^2=x^3+\frac{D}{4}$, then the Weierstrass equation of $E_{L_{f^c}}$ is automatically $y^2=x^3+\frac{\bar{D}}{4}$. By comparing the Neron lattice we see that $D=\bar\varpi^i D'$ with $D'$ coprime to $p$. Then from \eqref{modularpara}, we know that $c$ is coprime to $p$ which proves $c$ is a unit in $\CO_K$. This also implies $c'$  the Manin-Stevens constant of $E_{p^{2i}}$ is also $1$.

 \section{Integrality of the modular parametrization}
 By \cite[Chapter VI, Theorem 3.5]{Silvermanbook1}, the Laurent series for $\wp_L(z)$ around $z=0$ is given by
$$\wp_L(z)=\frac{1}{z^2}+\sum_{k=1}^{\infty}(2 k+1) G_{2 k+2}(L)z^{2 k},$$
where 
\begin{equation}\label{es}
G_{2 k+2}(L)=\sum_{\substack{\omega \in L \\ \omega \neq 0}} \omega^{-2k-2}
\end{equation}
is the Eisenstein series of weight $2k+2$ for the lattice $L$. By the expression \eqref{es} and the fact that $L$ is invariant under multiplication by $\omega$, we see that $G_{2k+2}(L)=\omega^{-2k-2}G_{2k+2}(L)$. So $G_{2k+2}(L)=0$ unless $2k\equiv 1\mod 3$. As a result


\begin{equation}\label{wp}
\wp_L(z)=\frac{1}{z^2}+\sum_{k=0}^\infty(6k+5) G_{6k+6}(L)z^{6k+4},
\end{equation}

\begin{eqnarray}\label{wpd}
\frac{1}{2}\wp'_L(z)=-\frac{1}{z^3}+\frac{1}{2}\sum_{k=0}^\infty(6k+4)(6k+5)G_{6k+6}(L)z^{6k+3}.
\end{eqnarray}
\begin{equation}\label{form}
2\pi i\s\int_{i\infty}^z f(\tau)d\tau=\sum_{n=1}^\infty \frac{a_n}{n}q^n=:\hat{f}(q)
\end{equation}
Put \eqref{form} into \eqref{wp}, \eqref{wpd} and \eqref{modularpara}, we get the Fourier expansion $x(q)$ and $y(q)$ of $x(z)$ and $y(z)$ which satisfy the equation of $E_{\bar\varpi^i}$. The same story holds for $E_{\varpi}$. The following Proposition can be easily checked.
\begin{proposition}\label{minimal}
The minimal Weierstrass model of $E_{\bar\varpi^i}$ is given by
\begin{enumerate}
\item $y^2+y=x^3+\frac{{\bar\varpi}^{2i}-1}{4}$, if $\bar\varpi^i\equiv 1\mod 2$;
\item $y^2+\omega y=x^3+\frac{{\bar\varpi}^{2i}-\omega^2}{4}$, if $\bar\varpi^i\equiv \omega\mod 2$;
\item $y^2+\omega^2 y=x^3+\frac{{\bar\varpi}^{2i}-\omega}{4}$, if $\bar\varpi^i\equiv 1+\omega\mod 2$.
\end{enumerate}
with variable transformation $x\mapsto x$, $y\mapsto y+\frac{1}{2}, y+\frac{\omega}{2}, y+\frac{\omega^2}{2}$.
\end{proposition}

The proof of the following theorem is obtained with the discussion with K. Cesnavicius, B.Conrad and F.Calegari.  Similar arguments were used to prove that the Manin Constants are integral, see e.g. \cite[(1.6)]{Stevens}, \cite[Propostion 1.1]{Edix}, \cite[Proposition 4.2]{GL01} and \cite[Theorem 3.3]{ARS}. 
\begin{theorem}\label{integral}
Let $\Gamma$ be $\Gamma_1(N)$ or $\Gamma_0(N)$ and $E$ be an elliptic curve over a class number one field $L$ which is parametrized by $X_\Gamma$. Assume
$\pi: X_\Gamma\rightarrow E$ is chosen such that $\pi(\infty)=O$ and is given by $\tau\mapsto (X(q),Y(q))$  where $q=e^{2\pi i\tau}$ and $X(q),Y(q)$ satisfy the minimal Weierstrass equation of $E$ over $L$. Then $\frac{1}{X(q)}\in \CO_L[[q]]$ and $\frac{1}{Y(q)}\in \CO_L[[q]]$.
\end{theorem}
\begin{proof}
Let $\CX_\Gamma^0$ be the smooth locus of the minimal proper regular model $\CX_\Gamma$ of $X_\Gamma$ over $\Spec(\CO_L)$ and $\CE$ be the Neron model of $E$ over $\CO_L$. By the Neron mapping property, the parametrization $\pi$ extends to a morphism $\pi: \CX_1(N)^0\rightarrow \CE$ over $\CO_L$. It is known that the minimal Weierstrass model over $\CO_L$ agrees near its identity section with the Neron model, see \cite[Theorem 5.5]{Con}. It is also known that the formal completion of $\CX_1(N)$ along the unramified cusp $\infty$ is $\Spf(\CO_L[[q]])$ with local parameter $q=e^{2\pi iz}$ given by the Tate curve, i.e. \cite[Theorem 8.11.10]{KZ}. The analytic $q$-expansion of modular forms (functions) over $\BC$ at $\infty$ agrees with their algebraic restrictions to $\Spf(\CO_L[[q]])$ of Katz's algebraic interpretation of them \cite[Chapter 1]{Katz}\cite[VII]{DR}. 
But note that our parametrization is chosen that $\phi(\infty)=O$ and the local uniformizer of the minimal Weierstrass model of $E_{\bar\varpi^i}$ is given by the inverse of $x$ or $y$ \cite[Section IV.1]{Silvermanbook1}. This means that the $q$-expansions of $\frac{1}{X(q)}$ and $\frac{1}{Y(q)}$ at the cusp $\infty$ have coefficients in $\CO_L$.
\end{proof}
\begin{remark}
In the proof above, we use two models of the modular curves developed by Katz-Mazur and Deligne-Rapoport. They treated the cusp differently. Katz-Mazur treated the Tate curve as being defined by the explicit equation over $\BZ[[q]]$ without showing their modular curves have a moduli meaning near their notion of cusps. Deligne-Rapoport used the generalized elliptic curves to treat their cusp but does not show their abstract approach to the Tate curves over $\BZ[[q]]$. But Conrad showed that these two approaches agree in
\cite{Conrad07} and his forth coming paper of his talk at the conference in March 2025 at Harvard in honor of Tate's 100th birthday.
\end{remark}

Now combine \eqref{modularpara}, Proposition \ref{minimal} and Theorem \ref{integral}, we see that $\frac{c^2}{x(q)}\in\CO_K[[q]]$, $\frac{1}{\frac{y(q)}{c^3}+\frac{\omega^j}{2}}\in\CO_K[[q]]$ with $\bar\varpi^i\equiv \omega^j\mod 2$. As we prove that $c\in\CO_K^\times$ in section \ref{c} and the leading term of $y(q)$ is $\frac{-1}{q^3}$, we deduce that $\frac{y(q)}{c^3}-\frac{\bar\varpi^i}{2}\in q^{-3}\CO_K[[q]]$ with first coefficient in $\CO_K^\times$. It is easy to see the discussion in this section also works for $y^c(q)$.  Then we have proved

\begin{proposition}\label{yq}
Both $y(q)-\frac{\bar\varpi^i}{2}$ and $y^c(q)-\frac{\varpi^i}{2}$ belong to $q^{-3}\CO_K[[q]]$ with first coefficients in $\CO_K^\times$.
\end{proposition}

\section{Nontriviality of the CM points}

Let $\tau=\frac{-1}{3(\omega+r)}$ with $r^2-r+1\equiv 0$ be the CM point we consider. In this section, we prove that either $\varphi(\tau)$ or $\varphi^c(\tau)$ is non-torsion on the elliptic curve $E_{\bar{\varpi}^i}$ or $E_{{\varpi}^i}$.

Recall the functions in \eqref{modular-function},
\begin{equation}\label{id1}
\lrb{y(z)+\frac{\bar\varpi^i}{2}}\lrb{y^c(z)+\frac{\varpi^i}{2}}^{-1}=(F_+(z))^3,
\end{equation}
\begin{equation}\label{id2}
\lrb{y(z)-\frac{\bar\varpi^i}{2}}\lrb{y^c(z)-\frac{\varpi^i}{2}}^{-1}=(F_-(z))^3,
\end{equation}
\begin{equation}\label{id3}
\lrb{y(z)-\frac{\bar\varpi^i}{2}}\lrb{y^c(z)+\frac{\varpi^i}{2}}=(G(z))^3,
\end{equation}
\begin{equation}\label{id4}
\lrb{y(z)+\frac{\bar\varpi^i}{2}}\lrb{y^c(z)-\frac{\varpi^i}{2}}=(H(z))^3
\end{equation} 

\begin{proposition}\label{cubic} 
The modular functions $F_+(z),F_-(z),G(z),H(z)$ are invariant under the action of $\Gamma$. 
\end{proposition}
 \begin{proof}
 We only need prove the proposition for $F_+(z)$, since 
 $$F_-(z)=F_+(z)\frac{x^c(z)}{x(z)},\ \ H(z)=F_+(z)x^c(z),\ \ G(z)=F_-(z)x^c(z)$$ 
 and both $x(z)$ and $x^c(z)$ are $\Gamma$ invariant.
 
First, we reduce the proof to prove that $F_+(z)$ are $\Gamma_1(N)$ invariant. If $F_+(z)$ are $\Gamma_1(N)$ invariant, then the action of $\Gamma_0(N)$ on $F_+(z)$ factor through the abelian group $\Gamma_0(N)/\Gamma_1(N)=(\BZ/N\BZ)^\times$. Since $\Gamma_0(N)$ acts on $F_+(z)$ by a cubic character, the action of $\Gamma_0(N)/\Gamma_1(N)$ factor through the cube elements in $\Gamma_0(N)/\Gamma_1(N)$. From this we can see $F_+(z)$ are $\Gamma$ invariant.

Second, since $\Gamma(N)$ is a normal subgroup of $\Gamma_1(N)$ and $\Gamma_1(N)/\Gamma(N)=\set{I,T,T^2,\cdots,T^{N-1}}$.
It is easy to see that $T$ acts on these functions trivially, it is enough to prove the proposition for $\Gamma(N)$. By a theorem of Fricke and Wohlfahrt \cite{Woh}(see also \cite[Page 526, Lemma]{Glenn}), $\Gamma(N)$ is generated by $\Gamma(N')$ and the parabolic elements of $\Gamma(N)$ for any $N'\mid N$. The parabolic elements also act trivially on $F_+(z)$. We just need to prove that $F_+(z)$ is modular on some congruence subgroup. Since 
\[\CY(z)=\frac{y(z)+\frac{\bar\varpi^i}{2}}{y^c(z)+\frac{\varpi^i}{2}}\]
is $\Gamma_0(N/3)$ invariant, it is easy to know that $F_+(z)$ is modular by some index-3 subgroup $\Gamma''$ of $\Gamma_0(N/3)$. We will prove $\Gamma''$ is congruence using the Unbounded Denominator Conjecture proved in \cite{CDT}.

By Proposition \ref{yq}, the Fourier expansion $\CY(q)$ of $\CY(z)$ belongs to $\CO_K[[q]]$ with the first coefficient invertible in $\CO_K$. Since the Fourier coefficients of $f(\tau)$ are complex conjugates of the Fourier coefficients of $f^c(\tau)$, their Fourier expansions satisfy $\hat{f}(q)\equiv \hat{f^c}(q) \mod \sqrt{-3}$ in \eqref{form}(remind that $a_n\neq 0$ only if $n\equiv 1\mod 3$). So $\hat{f}(q)^3\equiv \hat{f^c}(q)^3\mod 3\sqrt{-3}$. Bringing this into \eqref{wpd}, we see $y(q)\equiv y^c(q)\mod 3\sqrt{-3}$. Since $\frac{\bar\varpi^i}{2}\equiv \frac{\varpi^i}{2}\mod 3$, we finally get
$$y(q)+\frac{\bar\varpi^i}{2}\equiv y^c(q)+\frac{\varpi^i}{2}\mod 3$$
 and so $\CY(q)\equiv 1\mod 3$. Then the same method as \cite[Theorem 1]{HRS} shows $F_+(q)=\CY(q)^{1/3}$ also belongs to $\CO_K[[q]]$. This will imply $F_+(z)$ is modular on congruence subgroup by the main theorem of \cite{CDT}. This finishes the proof of the Proposition.
\end{proof}
\begin{remark}
Although the main theorem of \cite{CDT} is stated for holomorphic modular forms and modular functions, it is communicated by the authors that their results also hold for meromorphic modular forms and modular functions. So we can apply their result here.
\end{remark}

\begin{proposition}\label{torsion}
Let $r$ be an integer such that $r^2-r+1\equiv 0\mod 3p$. If $-r\equiv \omega\mod\varpi$ and  $\varphi^c(\tau_r)$ is a torsion point on $E_{\varpi^i}$, then $\varphi(\tau_r)\neq O$. If $-r\equiv \omega^2\mod\varpi$ and $\varphi(\tau_r)$ is a torsion point on $E_{\bar\varpi^i}$, then $\varphi^c(\tau_r)\neq O$.
\end{proposition}
\begin{proof}
We only prove the first case, second case can be proved similarly.

If $-r\equiv \omega\mod\varpi$, we know that $\varphi(\tau_r)\in E_{\bar\varpi^i}(K)$ should be a torsion point by Theorem \ref{cmpoint} . 

By the choice of $r$, the matrix $\iota_1(\omega)$ fixes $\tau_r$ and acts on $f(z)$ (resp. $f^c(z)$) as $f(z)\mid_{\iota_1(\omega)}=\omega f(z)$ (resp. $f^c(z)\mid_{\iota_1(\omega)}=f^c(z)$) see Proposition \ref{AB} and \eqref{omega1}, \eqref{omega2}.  So 
\[\int_{i\infty}^{\iota_1(\omega)\infty}f^c(\tau)d\tau=\int_{i\infty}^{\tau_r}f^c(\tau)d\tau+\int_{\tau_r}^{\iota_1(\omega)\infty}f^c(\tau)d\tau=0,\]
and
\[2\pi i\int_{i\infty}^{\iota_1(\omega)\infty}f(\tau)d\tau=(1-\omega)2\pi i\int_{i\infty}^{\tau_r}f(\tau)d\tau\in L.\]
Then by Proposition \ref{mk} and \ref{LK},

\begin{eqnarray*}
K\lrb{2\pi i\int_{i\infty}^{\iota_1(\omega)z}f(\tau)d\tau}&=&K\lrb{2\pi i\int_{i\infty}^{\iota_1(\omega)\infty}f(\tau)d\tau +2\pi i\int_{\iota_1(\omega)\infty}^{\iota_1(\omega)z}f(\tau)d\tau}\\
&=&\nu\lrb{2\pi i\int_{i\infty}^{\iota_1(\omega)\infty}f(\tau)d\tau }\omega^{-1} K\lrb{2\pi i\int_{i\infty}^{z}f(\tau)d\tau}
\end{eqnarray*}
and 
\begin{eqnarray*}
K^c\lrb{2\pi i\int_{i\infty}^{\iota_1(\omega)z}f^c(\tau)d\tau}&=&K^c\lrb{2\pi i\int_{i\infty}^{\iota_1(\omega)\infty}f^c(\tau)d\tau +2\pi i\int_{\iota_1(\omega)\infty}^{\iota_1(\omega)z}f^c(\tau)d\tau}\\
&=&K^c\lrb{2\pi i\int_{i\infty}^{z}f^c(\tau)d\tau}
\end{eqnarray*}

So we have 
\begin{equation}\label{F}
F_+(\iota_1(\omega)z)=\nu\lrb{2\pi i\int_{i\infty}^{\iota_1(\omega)\infty}f(\tau)d\tau }\omega^{-1}F_+(z).
\end{equation}
The same formula also holds for $F_-(z), H(z)$ and $G(z)$. Take $z=\tau_r$, if $\varphi(\tau_r)=O$, then $y(\tau_r)=\infty$. Under this assumption, if $y^c(\tau_r)=\infty$, then $F_+(\tau_r)\neq 0,\infty$; if $y^c(\tau_r)=\frac{\varpi^i}{2}$, then $H(\tau_r)\neq 0,\infty$; if $y^c(\tau_r)=-\frac{\varpi^i}{2}$; then $G(\tau_r)\neq 0,\infty$. In summary, if $\varphi(\tau_r)=O$, then at least one of $F_+(\tau_r), H(\tau_r)$ and $G(\tau_r)$ is not zero or infinity. Since $\iota_1(\omega)\tau_r=\tau_r$, from \eqref{F} and similar equations for $H(z), G(z)$, we get
\begin{equation}\label{actionr}
\nu\lrb{2\pi i\int_{i\infty}^{\iota_1(\omega)\infty}f(\tau)d\tau }=\omega.
\end{equation}
As a result we know
\[2\pi i\int_{i\infty}^{\iota_1(\omega)\infty}f(\tau)d\tau \notin \sqrt{-3}L.\]
Since 
$$(1-\omega)\int_{i\infty}^{\tau_r}f(\tau)d\tau=\int_{i\infty}^{\iota_1(\omega)\infty}f(\tau)d\tau,$$ 
we get
$$2\pi i\int_{i\infty}^{\tau_r}f(\tau)d\tau\in \frac{1}{\sqrt{-3}}L$$
but not in $L$. This contradict the assumption that $\varphi(\tau_r)=0$. This means $\varphi(\tau_r)$ is a primitive $\sqrt{-3}$ torsion point on $E_{\bar\varpi^i}$.
\end{proof}

Similar proofs using Theorem \ref{cmpoint1}, \eqref{omega12} and \eqref{omega32} give the following proposition.
\begin{proposition}\label{torsion1}
Let $r$ be an integer such that $r^2-r+1\equiv 0\mod 3p$. If $-r\equiv \omega\mod\varpi$ and  $\varphi(W(\tau_r))$ is a torsion point on $E_{\bar\varpi^i}$, then $\varphi^c(W(\tau_r))\neq O$. If $-r\equiv \omega^2\mod\varpi$ and $\varphi^c(W(\tau_r))$ is a torsion point on $E_{\varpi^i}$, then $\varphi(W(\tau_r))\neq O$.
\end{proposition}

\begin{remark}
By \eqref{action0} and \eqref{actionr},  we can see that the torsion point 
$$\varphi(\tau_r)=-\varphi(0).$$
\end{remark}

\begin{theorem}\label{non-torsion}
Let $r$ be an integer such that $r^2-r+1\equiv 0\mod 3p$. If $-r\equiv \omega^2\mod\varpi$, then $\varphi(\tau_r)$ and $\varphi^c(W(\tau_r))$ are non-torsion.  If $-r\equiv \omega\mod\varpi$, then $\varphi^c(\tau_r)$ and $\varphi(W(\tau_r))$ are non-torsion. 
\end{theorem}

\begin{proof}By \eqref{involution}, we have
\begin{equation}\label{cw}
2\pi i\int_{\infty}^{\tau_r}f^c(z)dz=\frac{2\pi i}{C}\int_{\infty}^{W(\tau_r)}f(z) dz-\frac{2\pi i}{C}\int_{\infty}^{0}f(z)dz,
\end{equation}
where $C$ is the constant in \eqref{involution}.
By Proposition \ref{torsion0}, $\varphi(0)\neq O$ and $\varphi^c(0)\neq O$. Then from \eqref{cw}, we see that at least one of $\varphi^c(\tau_r)$ and $\varphi(W(\tau_r))$ is not $O$. Then by Proposition \ref{torsion} and \ref{torsion1}, 
Both $\varphi(\tau_r)$ and $\varphi^c(\tau_r)$ are not $O$ or Both $\varphi(W(\tau_r))$ and $\varphi^c(W(\tau_r))$ are not $O$. Without loss of generality, we can assume $\varphi(\tau_r)\neq O$ and $\varphi^c(\tau_r)\neq O$. 

Assume $\varphi(\tau_r)$ and $\varphi^c(\tau_r)$ are both torsion. 
Similar to Theorem \ref{cmpoint}, $F_+(\tau_r),F_-(\tau_r),G(\tau_r),H(\tau_r)$ are all defined over $K(\sqrt[3]{\varpi})$ or $K(\sqrt[3]{\bar\varpi})$ depending on $-r\equiv \omega^2\mod\varpi$ or $-r\equiv \omega\mod\varpi$. Since $\varphi(\tau_r)$ and $\varphi^c(\tau_r)$ are non-zero $\sqrt{-3}$-torsion points, then $y(\varphi(\tau_r))=\frac{\bar\varpi^i}{2}$ or $-\frac{\bar\varpi^i}{2}$ while $y(\varphi^c(\tau_r))=\frac{\varpi^i}{2}$ or $-\frac{\varpi^i}{2}$. If $y(\varphi(\tau_r))=\frac{\bar\varpi^i}{2}$ and $y(\varphi^c(\tau_r))=\frac{\varpi^i}{2}$, then by equation (\ref{id1}), $\frac{\bar\varpi^i}{\varpi^i}$ is a cube in $K(\sqrt[3]{\varpi})$ or $K(\sqrt[3]{\bar\varpi})$, but this is impossible. Using other three equations, we can show that other three cases of $y(\varphi(\tau_r))$ and $y(\varphi^c(\tau_r))$ also cannot happen. So $\varphi(\tau_r)$ and $\varphi^c(\tau_r)$ are not both torsion points. By \eqref{cw}, $\varphi(W(\tau_r))$ and $\varphi^c(W(\tau_r))$ are also not both torsion points. Then the theorem follows from Theorem \ref{cmpoint} and Theorem \ref{cmpoint1}.
\end{proof}

\begin{theorem}
Let $p\equiv 4,7\mod 9$ be a prime, then both $p$ and $p^2$ are sums of two rational cubes. 
\end{theorem}
\begin{proof}
For $i=1,2$, we consider the following maps:
\begin{equation}\label{twist-map}
\xymatrix{X_{\Gamma_1(N)}\ar[r]^{\varphi}& E_{\bvarpi}\ar[r]^{\phi}& E_{p^i}}
\end{equation}
where $\phi$ is the twisting map given by
\[(x,y)\mapsto (\sqrt[3]{\varpi^{2i}}x, \varpi^i y).\]
Let $r\in\BZ$ such that $r^2-r+1\equiv 0\mod 3p$ and $-r\equiv \omega^2\mod\varpi$. Then by Proposition \ref{LCF}, Theorem \ref{cmpoint} and Theorem \ref{non-torsion}, 
\[\phi(\varphi(\tau_r))\in E_{p^i}(K)\]
is non-torsion. If the complex conjugate $\ov{\phi(\varphi(\tau_r))}=-\phi(\varphi(\tau_r))+T$ for some torsion point $T$ then $[\sqrt{-3}]\phi(\varphi(\tau_r))\in E_{p^i}(\BQ)$ is non-torsion. Otherwise $\ov{\phi(\varphi(\tau_r))}+\phi(\varphi(\tau_r))$ belongs to $E_{p^i}(\BQ)$ is non-torsion.
\end{proof}

\section{A conditional proof of the main theorem}
In this section we will show that Artin's primitive root conjecture for arithmetic progressions or the Generalized Riemann Hypothesis (GRH) for number fields also imply the main theorem in this paper. 

In 1927, Artin conjectured that any integer $n$ which is not $-1$ or a perfect square is a primitive root for infinitely many primes. This conjecture is quite open but Hooley proved it under the GRH for some number fields \cite{Hooley} in 1967. The most well known unconditional result is due to Heath-Brown. He proved there are at most two primes that do not satisfy Artin's conjecture \cite{HB} improving Gupta-Murty's breakthrough result in \cite{GM}. 
But we even do not know a concrete number satisfying Artin's conjecture. We also do not know whether for every primes $p$ there exists at least one prime q such that $p$ is a primitive root of $q$.
There are many kinds of generalizations of Artin's conjecture, see Moree's survey \cite{Moree}. Here we just mention its generalization to arithmetic progressions by Lenstra \cite{Lenstra}, i.e. every integer which is not $-1$ or a perfect squares should be a primitive root for infinitely many primes in every suitable arithmetic progressions without obvious obstructions. Lenstra also proved that GRH implies Artin's conjecture for arithmetic progressions, see also \cite{Moree1} for an explicit asymptotic formula. In this section we will make the following assumption:
\begin{enumerate}
\item[$(\star)$] For $p\equiv 4,7\mod9$ a prime, Let $N=9p$ or $27p$. For every  integer $d\equiv 1\mod N, c\equiv 0\mod N$,  there always exists a number $d'$ in the progression $\set{b+cn\mid n\in\BZ}$ such that $p$ is a primitive root mod $d'$.
\end{enumerate}
Obviously, Artin's conjecture for arithmetic progressions or the GRH imply the assumption $(\star)$.

In the following we prove that Proposition \ref{cubic} can be deduced from $(\star)$. First of all, we need some preparation. For any $\tau\in\BH\cup\BQ$,
\begin{equation}\label{shift}
\int_{0}^{\tau+n}f(z)dz=\int_{0}^{\infty}f(z)dz+\int_{\infty}^{\tau+n}f(z)dz=\int_{0}^{\tau}f(z)dz
\end{equation}

\begin{proposition}
Let $\gamma=\matrixx{a}{b}{c}{d}\in\Gamma_1(N)$. Assume $(\star)$, then there exist integers $l_0$ and $l_1$ such that 
\[\int_{0}^{p^{l_0}\gamma' 0}f(\tau)d\tau=\int_{0}^{-p^{l_0}\gamma' 0}f^c(\tau)d\tau=0.\]
where 
\[\gamma'=\gamma\matrixx{1}{l_1}{0}{1}.\]
\end{proposition}
\begin{proof}
By $(\star)$, there is an integer $l_1$ such that $p$ is a primitive root modulo $d'=d+l_1c$. Then for $b'=b+l_1a$, there exists $l_0$ such that
$p^{l_0}b'\equiv 1\mod d'$. In this case, using \eqref{shift}

\begin{eqnarray*}
\int_{0}^{\pm p^{l_0}\gamma' \cdot 0}f(\tau)d\tau&=&\int_{0}^{\frac{\pm 1}{b'}}f(\tau)d\tau\\
&=&\int_{0}^{\matrixx{\pm 1}{0}{d'-1}{\pm 1}\cdot 1}f(\tau)d\tau\\
&=&\int_{0}^{1}f(\tau)d\tau=0
\end{eqnarray*}

\end{proof}

\begin{proposition}\label{plus}
Let $\gamma=\matrixx{a}{b}{c}{d}\in \Gamma_1(N)$. Assume $(\star)$, then
\[2\pi i\lrb{\int_{i\infty}^{\frac{a}{c}}f(\tau)d\tau+\int_{i\infty}^{-\frac{a}{c}}f(\tau)d\tau}\in \sqrt{-3}L\]
\end{proposition}
\begin{proof}

Since 

\begin{equation}
\int_{0}^{\pm\gamma' 0}f(\tau)d\tau=\int_{0}^{\pm\gamma 0}f(\tau)d\tau+\int_{0}^{\pm l_1}f(\tau)d\tau=\int_{0}^{\pm\gamma 0}f(\tau)d\tau,
\end{equation}
Combining this with \eqref{integral-twist}, we know that
\begin{equation}
\int_{0}^{\pm\gamma 0}f(\tau)d\tau=\sum_{k=0}^{l_1-1}\frac{a_p^k}{p^k\tau(\bar\chi)}\sum_{u=1}^{p-1}\bar\chi(u)\int_{\pm\frac{u}{p}}^{\pm p^k\gamma 0\pm \frac{u}{p}}g(\tau)d\tau.
\end{equation}
So we have
\begin{eqnarray}\label{period-sum}
&&\int_{0}^{\gamma 0}f(\tau)d\tau+\int_{0}^{-\gamma 0}f(\tau)d\tau \\ 
&=&\sum_{k=0}^{l_0-1}\frac{a_p^k}{p^k\tau(\bar\chi)}\sum_{u=1}^{p-1}\bar\chi(u)\lrb{\int_{\frac{u}{p}}^{p^k\gamma' 0 \frac{u}{p}}g(\tau)d\tau+\int_{-\frac{u}{p}}^{-p^k\gamma' 0-\frac{u}{p}}g(\tau)d\tau}\notag
\end{eqnarray}
Since the Manin-Stevens constant of $E_{p^{2i}}$ is prove to be $1$ in the end of section \ref{c}, by \cite[(2.9)]{Stevens}, the optimal quotient for $J_1(Np)$ in the isogeny class of $E_{p^{2i}}$ is $E_{p^{2i}}$. From  the discussion in \cite[Section 5]{Grossbook}, we know that the period lattice $L_g$ of $g$ is generated by a pure imaginary period and any real period of $g$ belongs to $\sqrt{-3}L_g$. We get
$$2\pi i\int_{\frac{u}{p}}^{p^k\gamma' 0 \frac{u}{p}}g(\tau)d\tau+2\pi i\int_{-\frac{u}{p}}^{-p^k\gamma' 0-\frac{u}{p}}g(\tau)d\tau\in \sqrt{-3}L_g$$
since the left hand is a real number, see \cite[2.1.3]{Cremonabook}. From \eqref{period-sum}, we see that 
$$p^{l_0}\lrb{2\pi i\int_{0}^{\gamma 0}f(\tau)d\tau+2\pi i\int_{0}^{-\gamma 0}f(\tau)d\tau}\in\sqrt{-3}L.$$
which implies
$$2\pi i\int_{0}^{\gamma 0}f(\tau)d\tau+2\pi i\int_{0}^{-\gamma 0}f(\tau)d\tau\in\sqrt{-3}L$$
since $p$ is coprime to $\sqrt{-3}$.
\end{proof}

\begin{proof}[A proof of Proposition \ref{cubic} assuming $(\star)$]
Let $\gamma=\matrixx{a}{b}{c}{d}\in \Gamma_1(N)$. By \cite[2.1.3]{Cremonabook}, 
\[2\pi i\int_{i\infty}^{\frac{a}{c}}f^c(\tau)d\tau=\ov{2\pi i\int_{i\infty}^{-\frac{a}{c}}f(\tau)d\tau}.\]
So if $$2\pi i\int_{i\infty}^{-\frac{a}{c}}f(\tau)d\tau$$ acts on $K(z)$ with eigenvalue $\omega^k$, then $$2\pi i\int_{i\infty}^{\frac{a}{c}}f^c(\tau)d\tau$$ will acts on $K^c(z)$ with eigenvalue $\bar\omega^k$. By Proposition \ref{plus}, $$2\pi i\int_{i\infty}^{\frac{a}{c}}f(\tau)d\tau$$ also acts on $K(z)$ with eigenvalues
$\bar\omega^k$. Then it is easy to see that 
$$F_+(z)=K\lrb{2\pi i\int_{i\infty}^{z}f(\tau)d\tau}/K^c\lrb{2\pi i\int_{i\infty}^{z}f^c(\tau)d\tau}$$
is $\Gamma_1(N)$ invariant. As in the proof in section 8, this is enough to prove Proposition \ref{cubic}.
\end{proof}

\section{Examples}

(1) The first example is $p=7$, $\varpi=3\omega+1$, $\bar\varpi=-3\omega-2$. Both the elliptic curves $E_{\varpi}$ and $E_{\bar\varpi}$ are parametrized by $X_1(189)$. Note our notation $\varphi$ is the parametrization of $E_{\bar\varpi}$ while $\varphi^c$ is for $E_{\varpi}$.
\begin{eqnarray*}
y(q)-\frac{\bar\varpi^2}{2}&=&y^c(q)-\frac{\varpi^2}{2}\\
&=&-q^{-3} + 1 + q^3 + q^6 - q^9 - 2q^{12} + q^{15} - 3q^{18} + q^{21} + q^{24}\\
&& + 2q^{27} - q^{33} + 2q^{36} - 4q^{39} + q^{42} + 3q^{45} + O(q^{47})
\end{eqnarray*}
So in this case, $F_-(q)=1$. Let $\tau=\frac{-1}{3(\omega+5)}$, then 
$$\varphi(\tau)=\lrb{(4\omega-1)\varpi^{-\frac{2}{3}},  \frac{-9\omega-4}{2}}$$
$$\varphi^c(\tau)=\lrb{0, -\frac{\varpi}{2}}$$
It is easy to check that 
\[{\frac{-9\omega-4}{2}-\frac{\bar\varpi}{2}}=-\frac{\varpi}{2}-\frac{\varpi}{2}.\]
compatible with the fact that $F_-(q)=1$. The CM point
$$\phi\circ\varphi(\tau)-\phi^c\circ\varphi^c(\tau)=\lrb{-\frac{7}{3}\omega,\frac{7}{18}\sqrt{-3}}$$
gives non-torsion points on $E_7$.

(2)Another example is $p=13$, $\varpi=3\omega+4$, $\bar\varpi=-3\omega+1$. Both the elliptic curves $E_{\varpi}$ and $E_{\bar\varpi}$ are parametrized by $X_1(117)$. 
\begin{eqnarray*}
y(q)+\frac{\bar\varpi}{2}&=&y^c(q)+\frac{\varpi}{2}\\
&=&-q^{-3} + 2 + q^3 - 2q^6 - q^9 - 2q^{12} + 2q^{15} + 2q^{21} + 2q^{24}\\
&& - q^{27} - 4q^{36} + q^{39} + 4q^{42} - 6q^{45} + O(q^{47})
\end{eqnarray*}
So in this case, $F_+(q)=1$. Let $\tau=\frac{-1}{3(\omega+23)}$, then 
$$\varphi(\tau)=\lrb{(-2\omega-7)\varpi^{-\frac{2}{3}},  \frac{9\omega+7}{2}}$$
$$\varphi^c(\tau)=\lrb{0, \frac{\varpi}{2}}$$
It is easy to check that 
\[{\frac{9\omega+7}{2}+\frac{\bar\varpi}{2}}=\frac{\varpi}{2}+\frac{\varpi}{2}.\]
compatible with the fact that $F_+(q)=1$. The CM point
$$\phi\circ\varphi(\tau)-\phi^c\circ\varphi^c(\tau)=\lrb{-\frac{13}{3}\omega^2,\frac{65}{18}\sqrt{-3}}$$
gives non-torsion points on $E_{13}$.

(3)Finally we give a little nontrivial example $p=31$, $\varpi=6\omega+1$, $\bar\varpi=-6\omega-5$. Both the elliptic curves $E_{\varpi}$ and $E_{\bar\varpi}$ are parametrized by $X_1(279)$. 
\begin{eqnarray*}
y(q)+\frac{\bar\varpi}{2}&=&-q^{-3} + (3\omega - 4) + q^3 + (6\omega + 1)q^6 + (3\omega + 8)q^9 + (3\omega + 1)q^{12} \\
&&+ 11q^{15} + 3\omega q^{18} + (-12\omega + 20)q^{21}+O(q^{22})
\end{eqnarray*}
\begin{eqnarray*}
y^c(q)+\frac{\varpi}{2}&=&-q^{-3} + (-3\omega-1) + q^3 + (-6\omega - 5)q^6 + (-3\omega + 5)q^9 + (-3\omega-2)q^{12} \\
&&+ 11q^{15} + (-3\omega-3) q^{18} + (12\omega + 32)q^{21}+O(q^{22})
\end{eqnarray*}

\begin{eqnarray*}
\frac{y(q)+\frac{\bar\varpi}{2}}{y^c(q)+\frac{\varpi}{2}}&=&1 + (6\omega + 3)q^3 + (-15\omega - 21)q^6 + (-9\omega + 63)q^9 + (192\omega - 39)q^{12}\\
 &&+ (-750\omega - 429)q^{15} + (1458\omega + 2133)q^{18} + (90\omega - 5274)q^{21}+O(q^{22})
\end{eqnarray*}
In this case, 
\begin{eqnarray*}
F_+(q)&=&1 + (2\omega + 1)q^3 + (-5\omega - 4)q^6 + (5\omega + 10)q^9 + (\omega - 16)q^{12} \\
&&+ (-40\omega + 4)q^{15} + (109\omega + 65)q^{18} + (-165\omega - 240)q^{21}+O(q^{22}) 
\end{eqnarray*}

Let $\tau=\frac{-1}{3(\omega+26)}$, then 
$$\varphi(\tau)=\lrb{\frac{-130\omega-404}{49}\varpi^{-\frac{2}{3}},  -\frac{549}{343}\omega-\frac{2531}{686}}$$
$$\varphi^c(\tau)=\lrb{0, \frac{\varpi}{2}}$$
It  can be checked that 
\[
\lrb{-\frac{549}{343}\omega-\frac{2531}{686}+\frac{\bar\varpi}{2}}\lrb{\frac{\varpi}{2}+\frac{\varpi}{2}}^{-1}=\varpi.
\]
compatible with the fact that $F_+(\tau)\in K(\sqrt[3]{\varpi})$. The twist of CM point
$$\phi\circ\varphi(\tau)-\phi^c\circ\varphi^c(\tau)=\lrb{-\frac{217}{12}\omega,-\frac{3131}{72}\sqrt{-3}}$$
gives non-torsion points on $E_{31}$.

\bibliographystyle{alpha}
\bibliography{reference}

@book{Cox89,
	Author = {Cox, D. A.},
	Publisher = {John Wiley $\&$ Sons Inc.},
	Title = {Primes of the Form $x^2+n y^2$},
	Year = {1989}}

@article{DV17,
	Author = {Dasgupta, S. and Voight, J.},
	Journal = {Proc. Amer. Math. Soc. },
	Pages = {3257-3273 },
	Publisher = {American Mathematical Soc.},
	Title = {Sylvester's problem and mock Heegner points},
	Volume = {146},
	Year = {2018}}

@article{CST17,
	Author = {Cai, L. and Shu, J. and Tian, Y.},
	Journal = {Amer. Jour. of Math.},
	Number = {3},
	Pages = {785--816},
	Title = {Cube sum problem and an explicit {G}ross-{Z}agier formula},
	Volume = {139},
	Year = {2017}}

@article{Selmer51,
	Author = {Selmer, E. S.},
	Journal = {Acta Math.},
	Language = {fre},
	Pages = {203--362},
	Publisher = {Secr{\'e}tariat math{\'e}matique},
	Title = {The {D}iophatine equation $ax^3+by^3+cz^3=0$},
	Url = {http://eudml.org/doc/110749},
	Volume = {87},
	Year = {1951}}

@book{Silvermanbook2,
	Author = {Silverman, J. H.},
	Doi = {10.1007/978-1-4612-0851-8},
	Isbn = {0-387-94328-5},
	Mrclass = {11G05 (11G07 11G15 11G40 14H52)},
	Mrnumber = {1312368},
	Mrreviewer = {Henri Darmon},
	Pages = {xiv+525},
	Publisher = {Springer-Verlag, New York},
	Series = {Graduate Texts in Mathematics},
	Title = {Advanced topics in the arithmetic of elliptic curves},
	Url = {http://dx.doi.org/10.1007/978-1-4612-0851-8},
	Volume = {151},
	Year = {1994}}

@book{Silvermanbook1,
	Author = {Silverman, J. H.},
	Isbn = {0-387-96203-4},
	Mrclass = {11G05 (14H52)},
	Mrnumber = {1329092},
	Note = {Corrected reprint of the 1986 original},
	Pages = {xii+400},
	Publisher = {Springer-Verlag, New York},
	Series = {Graduate Texts in Mathematics},
	Title = {The arithmetic of elliptic curves},
	Volume = {106},
	Year = {1992}}

@book{Iwanbook,
	Author = {Iwaniec, Henryk},
	Doi = {10.1090/gsm/017},
	Isbn = {0-8218-0777-3},
	Mrclass = {11Fxx (11-02)},
	Mrnumber = {1474964},
	Mrreviewer = {B. Ramakrishnan},
	Pages = {xii+259},
	Publisher = {American Mathematical Society, Providence, RI},
	Series = {Graduate Studies in Mathematics},
	Title = {Topics in classical automorphic forms},
	Url = {http://dx.doi.org/10.1090/gsm/017},
	Volume = {17},
	Year = {1997}}

@book{Shimurabook,
	Author = {Shimura, G.},
	Isbn = {0-691-08092-5},
	Mrclass = {11Fxx (11-02 11G05 11G40)},
	Mrnumber = {1291394},
	Note = {Reprint of the 1971 original, Kan{\^o} Memorial Lectures, 1},
	Pages = {xiv+271},
	Publisher = {Princeton University Press, Princeton, NJ},
	Series = {Publications of the Mathematical Society of Japan},
	Title = {Introduction to the arithmetic theory of automorphic functions},
	Volume = {11},
	Year = {1994}}

@article{DV1,
	Author = {Dasgupta, S. and Voight, J.},
	Journal = {Arithmetic Geometry: Clay Mathematics Institute Summer School, Arithmetic Geometry, July 17-August 11, 2006, Mathematisches Institut, Georg-August-Universit{\"a}t, G{\"o}ttingen, Germany},
	Pages = {91},
	Publisher = {American Mathematical Soc.},
	Title = {Heegner points and Sylvester's conjecture},
	Volume = {8},
	Year = {2009}}

@article {BS,
    AUTHOR = {Birch, B. J. and Stephens, N. M.},
     TITLE = {The parity of the rank of the {M}ordell-{W}eil group},
   JOURNAL = {Topology},
  FJOURNAL = {Topology. An International Journal of Mathematics},
    VOLUME = {5},
      YEAR = {1966},
     PAGES = {295--299},
      ISSN = {0040-9383},
   MRCLASS = {14.40 (10.12)},
  MRNUMBER = {201379},
MRREVIEWER = {J. W. S. Cassels},
       DOI = {10.1016/0040-9383(66)90021-8},
       URL = {https://doi.org/10.1016/0040-9383(66)90021-8},
}

@article {HSY19,
    AUTHOR = {Hu, Y. and Shu, J. and Yin, H.},
     TITLE = {An explicit {G}ross-{Z}agier formula related to the
              {S}ylvester conjecture},
   JOURNAL = {Trans. Amer. Math. Soc.},
  FJOURNAL = {Transactions of the American Mathematical Society},
    VOLUME = {372},
      YEAR = {2019},
    NUMBER = {10},
     PAGES = {6905--6925},
      ISSN = {0002-9947},
   MRCLASS = {11G40 (11G05)},
  MRNUMBER = {4024542},
       DOI = {10.1090/tran/7760},
       URL = {https://doi.org/10.1090/tran/7760},
}

@article {Tian,
    AUTHOR = {Tian, Y.},
     TITLE = {Congruent numbers and {H}eegner points},
   JOURNAL = {Camb. J. Math.},
  FJOURNAL = {Cambridge Journal of Mathematics},
    VOLUME = {2},
      YEAR = {2014},
    NUMBER = {1},
     PAGES = {117--161},
      ISSN = {2168-0930},
   MRCLASS = {11G40 (11F11 11G05)},
  MRNUMBER = {3272014},
MRREVIEWER = {Steven Joel Miller},
       DOI = {10.4310/CJM.2014.v2.n1.a4},
       URL = {https://doi.org/10.4310/CJM.2014.v2.n1.a4},
}

@article{Sage,
	Author = {Sagemath},
	Journal={https://cloud.sagemath.com/projects},
	Title = {},
	Year={}
}

@article {SSY,
    AUTHOR = {Shu, Jie and Song, Xu and Yin, Hongbo},
     TITLE = {Cube sums of the forms {$3p$} and {$3p^2$}},
   JOURNAL = {Math. Z.},
  FJOURNAL = {Mathematische Zeitschrift},
    VOLUME = {299},
      YEAR = {2021},
    NUMBER = {3-4},
     PAGES = {2297--2325},
      ISSN = {0025-5874},
   MRCLASS = {11G05 (11G15 11G40)},
  MRNUMBER = {4329288},
MRREVIEWER = {Steven Joel Miller},
       DOI = {10.1007/s00209-021-02730-w},
       URL = {https://doi.org/10.1007/s00209-021-02730-w},
}

@article {NM1,
    AUTHOR = {Murabayashi, N.},
     TITLE = {On the field of definition for modularity of {CM} elliptic
              curves},
   JOURNAL = {J. Number Theory},
  FJOURNAL = {Journal of Number Theory},
    VOLUME = {108},
      YEAR = {2004},
    NUMBER = {2},
     PAGES = {268--286},
      ISSN = {0022-314X},
   MRCLASS = {11G15 (11G05)},
  MRNUMBER = {2098639},
MRREVIEWER = {Benjamin V. Howard},
       DOI = {10.1016/j.jnt.2004.05.004},
       URL = {https://doi.org/10.1016/j.jnt.2004.05.004},
}

@article {Weil49,
    AUTHOR = {Weil, Andr\'{e}},
     TITLE = {Numbers of solutions of equations in finite fields},
   JOURNAL = {Bull. Amer. Math. Soc.},
  FJOURNAL = {Bulletin of the American Mathematical Society},
    VOLUME = {55},
      YEAR = {1949},
     PAGES = {497--508},
      ISSN = {0002-9904},
   MRCLASS = {10.0X},
  MRNUMBER = {29393},
MRREVIEWER = {L. Carlitz},
       DOI = {10.1090/S0002-9904-1949-09219-4},
       URL = {https://doi.org/10.1090/S0002-9904-1949-09219-4},
}

@book {IRbook,
    AUTHOR = {Ireland, Kenneth F. and Rosen, Michael I.},
     TITLE = {A classical introduction to modern number theory},
    SERIES = {Graduate Texts in Mathematics},
    VOLUME = {84},
      NOTE = {Revised edition of {{\i}t Elements of number theory}},
 PUBLISHER = {Springer-Verlag, New York-Berlin},
      YEAR = {1982},
     PAGES = {xiii+341},
      ISBN = {0-387-90625-8},
   MRCLASS = {12-01 (10-01)},
  MRNUMBER = {661047},
MRREVIEWER = {Lindsay N. Childs},
}

@article {Shimura71,
    AUTHOR = {Shimura, Goro},
     TITLE = {On elliptic curves with complex multiplication as factors of
              the {J}acobians of modular function fields},
   JOURNAL = {Nagoya Math. J.},
  FJOURNAL = {Nagoya Mathematical Journal},
    VOLUME = {43},
      YEAR = {1971},
     PAGES = {199--208},
      ISSN = {0027-7630},
   MRCLASS = {10D25 (14H05)},
  MRNUMBER = {296050},
MRREVIEWER = {William F. Hammond},
       URL = {http://projecteuclid.org/euclid.nmj/1118798376},
}

@article {GL ,
    AUTHOR = {Gonz\'{a}lez, Josep and Lario, Joan-C.},
     TITLE = {Modular elliptic directions with complex multiplication (with
              an application to {G}ross's elliptic curves)},
   JOURNAL = {Comment. Math. Helv.},
  FJOURNAL = {Commentarii Mathematici Helvetici. A Journal of the Swiss
              Mathematical Society},
    VOLUME = {86},
      YEAR = {2011},
    NUMBER = {2},
     PAGES = {317--351},
      ISSN = {0010-2571},
   MRCLASS = {11G18 (11F11 11G15 11R37)},
  MRNUMBER = {2775131},
MRREVIEWER = {\'{A}lvaro Lozano-Robledo},
       DOI = {10.4171/CMH/225},
       URL = {https://doi.org/10.4171/CMH/225},
}

@book {Grossbook,
    AUTHOR = {Gross, Benedict H.},
     TITLE = {Arithmetic on elliptic curves with complex multiplication},
    SERIES = {Lecture Notes in Mathematics},
    VOLUME = {776},
      NOTE = {With an appendix by B. Mazur},
 PUBLISHER = {Springer, Berlin},
      YEAR = {1980},
     PAGES = {iii+95},
      ISBN = {3-540-09743-0},
   MRCLASS = {10D25 (14K22)},
  MRNUMBER = {563921},
MRREVIEWER = {A. N. Andrianov},
}

@article {AL,
    AUTHOR = {Atkin, A. O. L. and Li, Wen Ch'ing Winnie},
     TITLE = {Twists of newforms and pseudo-eigenvalues of {$W$}-operators},
   JOURNAL = {Invent. Math.},
  FJOURNAL = {Inventiones Mathematicae},
    VOLUME = {48},
      YEAR = {1978},
    NUMBER = {3},
     PAGES = {221--243},
      ISSN = {0020-9910},
   MRCLASS = {10D12},
  MRNUMBER = {508986},
MRREVIEWER = {R. A. Rankin},
       DOI = {10.1007/BF01390245},
       URL = {https://doi.org/10.1007/BF01390245},
}

@book {Miyake,
    AUTHOR = {Miyake, Toshitsune},
     TITLE = {Modular forms},
    SERIES = {Springer Monographs in Mathematics},
   EDITION = {English},
      NOTE = {Translated from the 1976 Japanese original by Yoshitaka Maeda},
 PUBLISHER = {Springer-Verlag, Berlin},
      YEAR = {2006},
     PAGES = {x+335},
      ISBN = {978-3-540-29592-1; 3-540-29592-5},
   MRCLASS = {11F11 (11F25 11F72)},
  MRNUMBER = {2194815},
}

@article {CW,
    AUTHOR = {Cremona, J. E. and Whitley, E.},
     TITLE = {Periods of cusp forms and elliptic curves over imaginary
              quadratic fields},
   JOURNAL = {Math. Comp.},
  FJOURNAL = {Mathematics of Computation},
    VOLUME = {62},
      YEAR = {1994},
    NUMBER = {205},
     PAGES = {407--429},
      ISSN = {0025-5718},
   MRCLASS = {11F67 (11F66 11G05 11G40)},
  MRNUMBER = {1185241},
MRREVIEWER = {Henri Darmon},
       DOI = {10.2307/2153419},
       URL = {https://doi.org/10.2307/2153419},
}

@article {Shimura73,
    AUTHOR = {Shimura, Goro},
     TITLE = {On the factors of the jacobian variety of a modular function
              field},
   JOURNAL = {J. Math. Soc. Japan},
  FJOURNAL = {Journal of the Mathematical Society of Japan},
    VOLUME = {25},
      YEAR = {1973},
     PAGES = {523--544},
      ISSN = {0025-5645},
   MRCLASS = {14H40 (10D05 12A70)},
  MRNUMBER = {318162},
MRREVIEWER = {A. N. Andrianov},
       DOI = {10.2969/jmsj/02530523},
       URL = {https://doi.org/10.2969/jmsj/02530523},
}

@article {Yin,
    AUTHOR = {Yin, Hongbo},
     TITLE = {On the {$8$} case of the {S}ylvester conjecture},
   JOURNAL = {Trans. Amer. Math. Soc.},
  FJOURNAL = {Transactions of the American Mathematical Society},
    VOLUME = {375},
      YEAR = {2022},
    NUMBER = {4},
     PAGES = {2705--2728},
      ISSN = {0002-9947},
   MRCLASS = {11G05},
  MRNUMBER = {4391731},
MRREVIEWER = {Jeremy A. Rouse},
       DOI = {10.1090/tran/8556},
       URL = {https://doi.org/10.1090/tran/8556},
}

@article {Stevens,
    AUTHOR = {Stevens, Glenn},
     TITLE = {Stickelberger elements and modular parametrizations of
              elliptic curves},
   JOURNAL = {Invent. Math.},
  FJOURNAL = {Inventiones Mathematicae},
    VOLUME = {98},
      YEAR = {1989},
    NUMBER = {1},
     PAGES = {75--106},
      ISSN = {0020-9910},
   MRCLASS = {11G05 (11F11)},
  MRNUMBER = {1010156},
MRREVIEWER = {Sheldon Kamienny},
       DOI = {10.1007/BF01388845},
       URL = {https://doi.org/10.1007/BF01388845},
}

@article {Yin1,
    AUTHOR = {Shu, Jie and Yin, Hongbo},
     TITLE = {Cube sums of the forms {$3p$} and {$3p^2$} {II}},
   JOURNAL = {Math. Ann.},
  FJOURNAL = {Mathematische Annalen},
    VOLUME = {385},
      YEAR = {2023},
    NUMBER = {3-4},
     PAGES = {1037--1060},
      ISSN = {0025-5831},
   MRCLASS = {11G05 (11G15 11G40)},
  MRNUMBER = {4566692},
MRREVIEWER = {Abhishek Juyal},
       DOI = {10.1007/s00208-022-02370-3},
       URL = {https://doi.org/10.1007/s00208-022-02370-3},
}

@incollection {Elkies,
    AUTHOR = {Elkies, Noam D.},
     TITLE = {Heegner point computations},
 BOOKTITLE = {Algorithmic number theory ({I}thaca, {NY}, 1994)},
    SERIES = {Lecture Notes in Comput. Sci.},
    VOLUME = {877},
     PAGES = {122--133},
 PUBLISHER = {Springer, Berlin},
      YEAR = {1994},
   MRCLASS = {11G05 (11Y35)},
  MRNUMBER = {1322717},
MRREVIEWER = {Fernando Rodr\'{\i}guez Villegas},
       DOI = {10.1007/3-540-58691-1\_49},
       URL = {https://doi.org/10.1007/3-540-58691-1_49},
}

@article {GL01,
    AUTHOR = {Gonz\'alez, Josep and Lario, Joan-C.},
     TITLE = {{$\bold Q$}-curves and their {M}anin ideals},
   JOURNAL = {Amer. J. Math.},
  FJOURNAL = {American Journal of Mathematics},
    VOLUME = {123},
      YEAR = {2001},
    NUMBER = {3},
     PAGES = {475--503},
      ISSN = {0002-9327,1080-6377},
   MRCLASS = {11G05 (11G18)},
  MRNUMBER = {1833149},
MRREVIEWER = {Ivica\ Gusi\'c},
       DOI = {10.1353/ajm.2001.0015},
       URL = {https://doi.org/10.1353/ajm.2001.0015},
}

@article {ABS,
    AUTHOR = {Alpöge, Levent and Bhargava, Manjul and Shnidman, Ari},
     TITLE = {Integers expressible as the sum of two rational cubes},
   JOURNAL = {arXiv:2210.10730v2},
  FJOURNAL = {},
    VOLUME = {with an appendix by Ashay Burungale and Christopher Skinner},
      YEAR = {},
    NUMBER = {},
     PAGES = {},
      ISSN = {},
   MRCLASS = {},
  MRNUMBER = {},
MRREVIEWER = {},
       DOI = {},
       URL = {},
}

@article {MS,
    AUTHOR = {Majumdar, Dipramit and Shingavekar, Pratiksha},
     TITLE = {Cube sum problem for integers having exactly two distinct
              prime factors},
   JOURNAL = {Proc. Indian Acad. Sci. Math. Sci.},
  FJOURNAL = {Indian Academy of Sciences. Proceedings. Mathematical
              Sciences},
    VOLUME = {133},
      YEAR = {2023},
    NUMBER = {2},
     PAGES = {Paper No. 43, 22},
      ISSN = {0253-4142},
   MRCLASS = {11G05 (11D25)},
  MRNUMBER = {4672790},
       DOI = {10.1007/s12044-023-00757-z},
       URL = {https://doi.org/10.1007/s12044-023-00757-z},
}

@article {MS1,
    AUTHOR = {Majumdar, Dipramit and Sury, B.},
     TITLE = {Cyclic cubic extensions of {$\Bbb{Q}$}},
   JOURNAL = {Int. J. Number Theory},
  FJOURNAL = {International Journal of Number Theory},
    VOLUME = {18},
      YEAR = {2022},
    NUMBER = {9},
     PAGES = {1929--1955},
      ISSN = {1793-0421},
   MRCLASS = {11R16 (11D25 11G05 12F05)},
  MRNUMBER = {4454460},
MRREVIEWER = {Abdelkader Zekhnini},
       DOI = {10.1142/S1793042122500993},
       URL = {https://doi.org/10.1142/S1793042122500993},
}

@article {Reyna,
    AUTHOR = {Arias de Reyna, Juan},
     TITLE = {Sylvester’s arithmetic problem},
   JOURNAL = {https://institucional.us.es/blogimus/en/2020/03/sylvesters-arithmetic-problem/},
  FJOURNAL = {},
    VOLUME = {},
      YEAR = {},
    NUMBER = {},
     PAGES = {},
      ISSN = {},
   MRCLASS = {},
  MRNUMBER = {},
MRREVIEWER = {},
       DOI = {},
       URL = {},
}

@article {Cremona,
    AUTHOR = {Cremona, J. E.},
     TITLE = {Modular symbols for {$\Gamma_1(N)$} and elliptic curves with
              everywhere good reduction},
   JOURNAL = {Math. Proc. Cambridge Philos. Soc.},
  FJOURNAL = {Mathematical Proceedings of the Cambridge Philosophical
              Society},
    VOLUME = {111},
      YEAR = {1992},
    NUMBER = {2},
     PAGES = {199--218},
      ISSN = {0305-0041},
   MRCLASS = {11F67 (11G05)},
  MRNUMBER = {1142740},
MRREVIEWER = {Philippe Satg\'{e}},
       DOI = {10.1017/S0305004100075307},
       URL = {https://doi.org/10.1017/S0305004100075307},
}

@article {Li75,
    AUTHOR = {Li, Wen Ch'ing Winnie},
     TITLE = {Newforms and functional equations},
   JOURNAL = {Math. Ann.},
  FJOURNAL = {Mathematische Annalen},
    VOLUME = {212},
      YEAR = {1975},
     PAGES = {285--315},
      ISSN = {0025-5831},
   MRCLASS = {10D05},
  MRNUMBER = {369263},
MRREVIEWER = {R. A. Rankin},
       DOI = {10.1007/BF01344466},
       URL = {https://doi.org/10.1007/BF01344466},
}

@article {MTT,
    AUTHOR = {Mazur, B. and Tate, J. and Teitelbaum, J.},
     TITLE = {On {$p$}-adic analogues of the conjectures of {B}irch and
              {S}winnerton-{D}yer},
   JOURNAL = {Invent. Math.},
  FJOURNAL = {Inventiones Mathematicae},
    VOLUME = {84},
      YEAR = {1986},
    NUMBER = {1},
     PAGES = {1--48},
      ISSN = {0020-9910},
   MRCLASS = {11G40 (11G05 14G10)},
  MRNUMBER = {830037},
MRREVIEWER = {Kenneth A. Ribet},
       DOI = {10.1007/BF01388731},
       URL = {https://doi.org/10.1007/BF01388731},
}

@book {Cremonabook,
    AUTHOR = {Cremona, J. E.},
     TITLE = {Algorithms for modular elliptic curves},
   EDITION = {Second},
 PUBLISHER = {Cambridge University Press, Cambridge},
      YEAR = {1997},
     PAGES = {vi+376},
      ISBN = {0-521-59820-6},
   MRCLASS = {11G05 (11F11 11F25 11G40 14G25)},
  MRNUMBER = {1628193},
MRREVIEWER = {Joseph H. Silverman},
}

@article {HRS,
    AUTHOR = {Heninger, Nadia and Rains, E. M. and Sloane, N. J. A.},
     TITLE = {On the integrality of {$n$}th roots of generating functions},
   JOURNAL = {J. Combin. Theory Ser. A},
  FJOURNAL = {Journal of Combinatorial Theory. Series A},
    VOLUME = {113},
      YEAR = {2006},
    NUMBER = {8},
     PAGES = {1732--1745},
      ISSN = {0097-3165,1096-0899},
   MRCLASS = {94B25 (05A15 11F27 11H71)},
  MRNUMBER = {2269551},
MRREVIEWER = {Simon\ N.\ Litsyn},
       DOI = {10.1016/j.jcta.2006.03.018},
       URL = {https://doi.org/10.1016/j.jcta.2006.03.018},
}

@article {CDT,
    AUTHOR = {Calegari, Frank and Dimitrov, Vesselin and Tang, Yunqing},
     TITLE = {The unbounded denominators conjecture},
   JOURNAL = {J. Amer. Math. Soc.},
  FJOURNAL = {Journal of the American Mathematical Society},
    VOLUME = {38},
      YEAR = {2025},
    NUMBER = {3},
     PAGES = {627--702},
      ISSN = {0894-0347,1088-6834},
   MRCLASS = {11F11},
  MRNUMBER = {4890656},
       DOI = {10.1090/jams/1053},
       URL = {https://doi.org/10.1090/jams/1053},
}

@article {Glenn,
    AUTHOR = {Stevens, Glenn},
     TITLE = {The cuspidal group and special values of {$L$}-functions},
   JOURNAL = {Trans. Amer. Math. Soc.},
  FJOURNAL = {Transactions of the American Mathematical Society},
    VOLUME = {291},
      YEAR = {1985},
    NUMBER = {2},
     PAGES = {519--550},
      ISSN = {0002-9947,1088-6850},
   MRCLASS = {11G16 (11F11 11G30 11G40 14G10)},
  MRNUMBER = {800251},
MRREVIEWER = {Kenneth\ A.\ Ribet},
       DOI = {10.2307/2000098},
       URL = {https://doi.org/10.2307/2000098},
}

@article {ARS,
    AUTHOR = {Agashe, Amod and Ribet, Kenneth and Stein, William A.},
     TITLE = {The {M}anin constant},
   JOURNAL = {Pure Appl. Math. Q.},
  FJOURNAL = {Pure and Applied Mathematics Quarterly},
    VOLUME = {2},
      YEAR = {2006},
    NUMBER = {2},
     PAGES = {617--636},
      ISSN = {1558-8599,1558-8602},
   MRCLASS = {11G40 (11F11 11G10)},
  MRNUMBER = {2251484},
MRREVIEWER = {Chandan\ Singh\ Dalawat},
       DOI = {10.4310/PAMQ.2006.v2.n2.a11},
       URL = {https://doi.org/10.4310/PAMQ.2006.v2.n2.a11},
}

@article {BF,
    AUTHOR = {Burungale, Ashay and Flach, Matthias},
     TITLE = {The conjecture of {B}irch and {S}winnerton-{D}yer for certain
              elliptic curves with complex multiplication},
   JOURNAL = {Camb. J. Math.},
  FJOURNAL = {Cambridge Journal of Mathematics},
    VOLUME = {12},
      YEAR = {2024},
    NUMBER = {2},
     PAGES = {357--415},
      ISSN = {2168-0930,2168-0949},
   MRCLASS = {11G40 (11G15 11R23)},
  MRNUMBER = {4779675},
       DOI = {10.4310/cjm.2024.v12.n2.a2},
       URL = {https://doi.org/10.4310/cjm.2024.v12.n2.a2},
}

@article{Ce,
	Author = {Cesnavicius, K.},
	Journal = {arXiv:1604.02165v1},
	Title = {The {M}anin-{S}tevens constant in the semistable case}
	}

@article {Ce1,
    AUTHOR = {Cesnavicius, K.},
     TITLE = {The {M}anin constant in the semistable case},
   JOURNAL = {Compos. Math.},
  FJOURNAL = {Compositio Mathematica},
    VOLUME = {154},
      YEAR = {2018},
    NUMBER = {9},
     PAGES = {1889--1920},
      ISSN = {0010-437X,1570-5846},
   MRCLASS = {11G05 (11G10 11G18 14G35)},
  MRNUMBER = {3867287},
MRREVIEWER = {Pete\ L.\ Clark},
       DOI = {10.1112/s0010437x18007273},
       URL = {https://doi.org/10.1112/s0010437x18007273},
}

@incollection {Edix,
    AUTHOR = {Edixhoven, Bas},
     TITLE = {On the {M}anin constants of modular elliptic curves},
 BOOKTITLE = {Arithmetic algebraic geometry ({T}exel, 1989)},
    SERIES = {Progr. Math.},
    VOLUME = {89},
     PAGES = {25--39},
 PUBLISHER = {Birkh\"{a}user Boston, Boston, MA},
      YEAR = {1991},
      ISBN = {0-8176-3513-0},
   MRCLASS = {11G05 (14G25)},
  MRNUMBER = {1085254},
MRREVIEWER = {Philippe\ Satg\'{e}},
       DOI = {10.1007/978-1-4612-0457-2\{_}3}

@article{Con,
	Author = {Conrad, B.},
	Journal = {https://math.stanford.edu/~conrad/papers/minimalmodel.pdf},
	Title = {Minimal models for elliptic curves}
	}

@book {KZ,
    AUTHOR = {Katz, Nicholas M. and Mazur, Barry},
     TITLE = {Arithmetic moduli of elliptic curves},
    SERIES = {Annals of Mathematics Studies},
    VOLUME = {108},
 PUBLISHER = {Princeton University Press, Princeton, NJ},
      YEAR = {1985},
     PAGES = {xiv+514},
      ISBN = {0-691-08349-5; 0-691-08352-5},
   MRCLASS = {11G05 (11F11 14G25 14K15)},
  MRNUMBER = {772569},
MRREVIEWER = {Kenneth\ A.\ Ribet},
       DOI = {10.1515/9781400881710},
       URL = {https://doi.org/10.1515/9781400881710},
}

@inproceedings {Katz,
    AUTHOR = {Katz, Nicholas M.},
     TITLE = {{$p$}-adic properties of modular schemes and modular forms},
 BOOKTITLE = {Modular functions of one variable, {III} ({P}roc. {I}nternat.
              {S}ummer {S}chool, {U}niv. {A}ntwerp, {A}ntwerp, 1972)},
    SERIES = {Lecture Notes in Math., Vol. 350},
     PAGES = {69--190},
 PUBLISHER = {Springer, Berlin-New York},
      YEAR = {1973},
   MRCLASS = {10D15 (14D20)},
  MRNUMBER = {447119},
MRREVIEWER = {V.\ V.\ Shokurov},
}

@inproceedings {DR,
    AUTHOR = {Deligne, P. and Rapoport, M.},
     TITLE = {Les sch\'{e}mas de modules de courbes elliptiques},
 BOOKTITLE = {Modular functions of one variable, {II} ({P}roc. {I}nternat.
              {S}ummer {S}chool, {U}niv. {A}ntwerp, {A}ntwerp, 1972)},
    SERIES = {Lecture Notes in Math., Vol. 349},
     PAGES = {143--316},
 PUBLISHER = {Springer, Berlin-New York},
      YEAR = {1973},
   MRCLASS = {14K10 (10D05)},
  MRNUMBER = {337993},
MRREVIEWER = {T.\ Oda},
}

@article {Hooley,
    AUTHOR = {Hooley, Christopher},
     TITLE = {On {A}rtin's conjecture},
   JOURNAL = {J. Reine Angew. Math.},
  FJOURNAL = {Journal f\"{u}r die Reine und Angewandte Mathematik. [Crelle's
              Journal]},
    VOLUME = {225},
      YEAR = {1967},
     PAGES = {209--220},
      ISSN = {0075-4102,1435-5345},
   MRCLASS = {10.66 (10.41)},
  MRNUMBER = {207630},
MRREVIEWER = {Werner\ G. H. Schaal},
       DOI = {10.1515/crll.1967.225.209},
       URL = {https://doi.org/10.1515/crll.1967.225.209},
}

@article {HB,
    AUTHOR = {Heath-Brown, D. R.},
     TITLE = {Artin's conjecture for primitive roots},
   JOURNAL = {Quart. J. Math. Oxford Ser. (2)},
  FJOURNAL = {The Quarterly Journal of Mathematics. Oxford. Second Series},
    VOLUME = {37},
      YEAR = {1986},
    NUMBER = {145},
     PAGES = {27--38},
      ISSN = {0033-5606,1464-3847},
   MRCLASS = {11A07 (11N13 11N35)},
  MRNUMBER = {830627},
MRREVIEWER = {D.\ J.\ Lewis},
       DOI = {10.1093/qmath/37.1.27},
       URL = {https://doi.org/10.1093/qmath/37.1.27},
}

@article {GM,
    AUTHOR = {Gupta, Rajiv and Murty, M. Ram},
     TITLE = {A remark on {A}rtin's conjecture},
   JOURNAL = {Invent. Math.},
  FJOURNAL = {Inventiones Mathematicae},
    VOLUME = {78},
      YEAR = {1984},
    NUMBER = {1},
     PAGES = {127--130},
      ISSN = {0020-9910,1432-1297},
   MRCLASS = {11A07},
  MRNUMBER = {762358},
MRREVIEWER = {Herbert\ Walum},
       DOI = {10.1007/BF01388719},
       URL = {https://doi.org/10.1007/BF01388719},
}

@article {Moree,
    AUTHOR = {Moree, Pieter},
     TITLE = {Artin's primitive root conjecture---a survey},
   JOURNAL = {Integers},
  FJOURNAL = {Integers},
    VOLUME = {12},
      YEAR = {2012},
    NUMBER = {6},
     PAGES = {1305--1416},
      ISSN = {1867-0652,1867-0660},
   MRCLASS = {11N37 (11A07 11B05)},
  MRNUMBER = {3011564},
       DOI = {10.1515/integers-2012-0043},
       URL = {https://doi.org/10.1515/integers-2012-0043},
}

@article {Lenstra,
    AUTHOR = {Lenstra, Jr., H. W.},
     TITLE = {On {A}rtin's conjecture and {E}uclid's algorithm in global
              fields},
   JOURNAL = {Invent. Math.},
  FJOURNAL = {Inventiones Mathematicae},
    VOLUME = {42},
      YEAR = {1977},
     PAGES = {201--224},
      ISSN = {0020-9910,1432-1297},
   MRCLASS = {12A05 (12A75)},
  MRNUMBER = {480413},
MRREVIEWER = {John\ V.\ Armitage},
       DOI = {10.1007/BF01389788},
       URL = {https://doi.org/10.1007/BF01389788},
}

@article {Moree1,
    AUTHOR = {Moree, Pieter},
     TITLE = {On primes in arithmetic progression having a prescribed
              primitive root. {II}},
   JOURNAL = {Funct. Approx. Comment. Math.},
  FJOURNAL = {Uniwersytet im. Adama Mickiewicza w Poznaniu. Wydzia\l
              Matematyki i Informatyki. Functiones et Approximatio
              Commentarii Mathematici},
    VOLUME = {39},
      YEAR = {2008},
    NUMBER = {part 1},
     PAGES = {133--144},
      ISSN = {0208-6573,2080-9433},
      ISBN = {978-83-232-1955-2},
   MRCLASS = {11N69 (11N13 11N36 11N56)},
  MRNUMBER = {2490093},
MRREVIEWER = {Francesco\ Pappalardi},
       DOI = {10.7169/facm/1229696559},
       URL = {https://doi.org/10.7169/facm/1229696559},
}

@article {Woh,
    AUTHOR = {Wohlfahrt, Klaus},
     TITLE = {An extension of {F}. {K}lein's level concept},
   JOURNAL = {Illinois J. Math.},
  FJOURNAL = {Illinois Journal of Mathematics},
    VOLUME = {8},
      YEAR = {1964},
     PAGES = {529--535},
      ISSN = {0019-2082},
   MRCLASS = {10.21 (20.65)},
  MRNUMBER = {167533},
MRREVIEWER = {J.\ Lehner},
       URL = {http://projecteuclid.org/euclid.ijm/1256059574},
}

@article{Monsky,
	Author = {Monsky, P.},
	Journal = {arXiv:2309.00162v1},
	Title = {The sum of two cubes problem—an approach that’s classroom friendly}
	}

@article {Conrad07,
    AUTHOR = {Conrad, Brian},
     TITLE = {Arithmetic moduli of generalized elliptic curves},
   JOURNAL = {J. Inst. Math. Jussieu},
  FJOURNAL = {Journal of the Institute of Mathematics of Jussieu. JIMJ.
              Journal de l'Institut de Math\'{e}matiques de Jussieu},
    VOLUME = {6},
      YEAR = {2007},
    NUMBER = {2},
     PAGES = {209--278},
      ISSN = {1474-7480,1475-3030},
   MRCLASS = {11G18 (11F25 14G22)},
  MRNUMBER = {2311664},
MRREVIEWER = {Mihran\ Papikian},
       DOI = {10.1017/S1474748006000089},
       URL = {https://doi.org/10.1017/S1474748006000089},
}
\end{document}